\documentclass{article}
\usepackage[a4paper,margin=0.81
in,top=3cm,bottom=3.5cm,footskip=0.5cm]{geometry}
\usepackage[a4paper]{geometry}
\usepackage[utf8]{inputenc}
\usepackage{amsmath,amsfonts,amssymb,amsthm,bbm,mathrsfs}
\usepackage[shortlabels]{enumitem}
\usepackage{xcolor}
\usepackage{graphicx}
\usepackage{tikz}
\usepackage{appendix}
\newcounter{assumption}
\renewcommand{\theassumption}{A\arabic{assumption}}
\usepackage{setspace}
\usepackage{fancyhdr}
\usepackage[
    backref=page
]{hyperref}

\hypersetup{
    colorlinks=true,
    linkcolor=blue!60!cyan,
    filecolor=magenta,      
    urlcolor=blue!60!cyan,
    citecolor=blue!60!cyan,
    pdfpagemode=FullScreen,
}

\newtheorem{thm}{Theorem}[section]

\newtheorem{rmk}{Remark}[section]

\newtheorem{prop}{Proposition}[section]
\newtheorem{lm}{Lemma}[section]

\numberwithin{equation}{section}

\fancypagestyle{firstpage}{
\fancyhf{} 

\fancyfoot[L]{%
\footnotesize
\rule{1.5cm}{0.8pt}\\
$^*$ Corresponding author.\\[0.15cm]
$^\dagger$ Cadi Ayyad University, National School of Applied Sciences,  
Laboratory of Mathematics, Modeling and Automatic Systems,  
B.P. 575, Marrakesh, Morocco. \\
$^\ddagger$ Dipartimento di Matematica, Università degli Studi di Salerno,  Via Giovanni Paolo II, 132, 84084 Fisciano (SA), Italy.  \\
E-mail addresses: \textcolor{blue!60!cyan}{\texttt{s.boulite@uca.ma}} (S. Boulite), \textcolor{blue!60!cyan}{\texttt{aelgrou@unisa.it}} (A. Elgrou), \textcolor{blue!60!cyan}{\texttt{arhandi@unisa.it}} (A. Rhandi).
}
}
\title{\textbf{Insensitizing Control Problems for Coupled Stochastic Parabolic Systems with State and Gradient Observations}}

\author{Said Boulite$^\dagger$, \,Abdellatif Elgrou$^\ddagger\,^*$ \,and\, Abdelaziz Rhandi$^\ddagger$}

\date{}

\begin{document}

\maketitle
\vspace{-2em}

\maketitle
\vspace{-3em}

\thispagestyle{firstpage}
\begin{center}
\end{center}
\begin{center}
\end{center}

\begin{abstract}
We study insensitizing control problems for a class of coupled linear stochastic parabolic systems. We establish the existence of controls such that a sentinel functional, involving localized observations of the state variables and their spatial gradients, is insensitive to small perturbations of the null initial data. We first reformulate the insensitizing control problem as a null controllability problem for a coupled forward--backward stochastic parabolic system, in which the observation terms induce both zeroth- and second-order coupling terms. By duality, the analysis is reduced to an observability inequality for the corresponding adjoint system. The main analytical contribution is the derivation of new global Carleman estimates for coupled stochastic parabolic systems with zeroth- and second-order coupling terms, under suitable geometric assumptions on the control and observation regions. These estimates yield the required observability inequalities and, consequently, the existence of insensitizing controls. Furthermore, depending on the value of a weighting parameter $\beta\in[0,1]$, which determines the relative contributions of the two state components to the sentinel functional, we consider two cases. If $\beta\in\{0,1\}$, the sentinel functional depends on only one state component, and a single localized control acting in the drift of the first equation is sufficient. In contrast, if $\beta\in(0,1)$, both state components contribute to the sentinel functional, and two localized controls acting in the drift terms of the two equations are sufficient. Moreover, the control strategy in this paper involves two additional controls acting throughout the diffusion terms.
\end{abstract}

\maketitle
\smallskip


\section{Introduction and Main Results}
Let \(G\subset\mathbb{R}^N\), \(N\geq1\), be a nonempty bounded domain
with smooth boundary \(\Gamma=\partial G\), and let \(T>0\). Denote by
\(G_i\), \(1\leq i\leq2\), the control regions and by
\(\mathcal{O}_i\), \(1\leq i\leq4\), the observation regions, all of which
are assumed to be nonempty open subsets of \(G\). For any measurable set \(S\), we denote by \(\chi_S\) its characteristic function. Throughout the paper, we use the following notation:

\[
Q = (0,T) \times G, \quad \text{and} \quad \Sigma = (0,T) \times \Gamma.
\]

Let \( (\Omega, \mathcal{F}, \{\mathcal{F}_t\}_{t \geq 0}, \mathbb{P}) \) be a fixed complete filtered probability space, on which a one-dimensional standard Brownian motion \( W(\cdot) \) is defined. Here, \( \{\mathcal{F}_t\}_{t \geq 0} \) is the natural filtration generated by \( W(\cdot) \), and augmented by all the \( \mathbb{P} \)-null sets in \( \mathcal{F} \). Let \( \mathcal{X} \) be a Banach space, and let \( C([0,T]; \mathcal{X}) \) denote the Banach space of all continuous \( \mathcal{X} \)-valued functions defined on \( [0,T] \). For some sub-\(\sigma\)-algebra \( \mathcal{G} \subset \mathcal{F} \), we define the Banach space \( L^2_{\mathcal{G}}(\Omega; \mathcal{X}) \) as the space of all \( \mathcal{X} \)-valued \( \mathcal{G} \)-measurable random variables \( X \) such that \( \mathbb{E} \left[ \|X\|_\mathcal{X}^2 \right] < \infty \), with the canonical norm. Additionally, we denote by \( L^2_\mathcal{F}(0,T; \mathcal{X}) \) the Banach space consisting of all \( \mathcal{X} \)-valued \( \{\mathcal{F}_t\}_{t \geq 0} \)-adapted stochastic processes \( X(\cdot) \) such that
$\mathbb{E} \left[ \| X(\cdot) \|_{L^2(0,T; \mathcal{X})}^2 \right] < \infty$, with the canonical norm. Similarly, \( L^\infty_\mathcal{F}(0,T; \mathcal{X}) \) is the Banach space of all \( \mathcal{X} \)-valued \( \{\mathcal{F}_t\}_{t \geq 0} \)-adapted essentially bounded stochastic processes, with the canonical norm. The space \( L^2_\mathcal{F}(\Omega; C([0,T]; \mathcal{X})) \)  is the Banach space consisting of all \( \mathcal{X} \)-valued \( \{\mathcal{F}_t\}_{t \geq 0} \)-adapted continuous stochastic processes \( X(\cdot) \) such that $\mathbb{E} \left[ \|X(\cdot)\|_{C([0,T]; \mathcal{X})}^2 \right] < \infty$,
with the canonical norm. 
Similarly, one can define the space $L^\infty_\mathcal{F}(\Omega;C^n([0,T];\mathcal{X}))$ for any positive integer $n$. 

Throughout this paper, we assume that the coefficients 
\(\sigma_{ij}^m : \Omega \times Q \to \mathbb{R}\) (\(i,j=1,2,\dots,N\); \,\(m=1,2\)) satisfy the following assumptions:

\begin{enumerate}[(1)]
    \item \(\sigma_{ij}^m=\sigma_{ij}^m(\omega, t, x) \in L^\infty_\mathcal{F}(\Omega;C^1([0,T];W^{2,\infty}(G)))\) and \(\sigma_{ij}^m = \sigma_{ji}^m\), for any \(1 \leq i,j \leq N\), \;\(m=1,2\).
    \item There exists a constant \(\sigma_0 > 0\) such that 
    \begin{align}\label{assmponalpha}
        \sum_{i,j=1}^N \sigma_{ij}^m(\omega, t, x) \kappa_i \kappa_j \geq \sigma_0 |\kappa|^2, \qquad \text{for any }\;\, (\omega, t, x, \kappa) \in \Omega \times Q \times \mathbb{R}^N,\quad m=1,2,
    \end{align}
    where \(x = (x_1, \dots, x_N)\) and \(\kappa = (\kappa_1, \dots, \kappa_N)\).
\end{enumerate}
For \(m=1,2\), we define the second-order differential operators \(\mathcal{L}_m\) by
\begin{align}\label{definofopeLi}
\mathcal{L}_m  = \sum_{i,j=1}^N \frac{\partial}{\partial x_i}\left(\sigma_{ij}^m(\omega, t, x)\frac{\partial }{\partial x_j}\right),\qquad (\omega, t, x)\in\Omega\times Q.
\end{align}

We introduce the following control spaces. The full control space is defined by

$$
\mathcal{U}
:=
L^2_{\mathcal F}(0,T;L^2(G_0))
\times
L^2_{\mathcal F}(0,T;L^2(G_1))
\times L^2_{\mathcal F}(0,T;L^2(G))\times L^2_{\mathcal F}(0,T;L^2(G)),
$$

whereas the control space associated with the first control region \(G_0\) is given by

$$
\mathcal{U}_0
:=
L^2_{\mathcal F}(0,T;L^2(G_0))
\times
L^2_{\mathcal F}(0,T;L^2(G))\times L^2_{\mathcal F}(0,T;L^2(G)).
$$

We consider the following coupled forward stochastic parabolic system
\begin{equation}\label{ass15}
\begin{cases}
\begin{array}{ll}
dy - \mathcal{L}_1 y \,dt = \Big[\xi_1 + a_{11} y + a_{12} z + B_{11}\cdot\nabla y + B_{12}\cdot\nabla z +  u_1 \chi_{G_0}\Big] \,dt + \Big[b_{11} y + b_{12} z + u_3\Big] \,dW(t) & \textnormal{in } Q, \\[0.8em]
dz - \mathcal{L}_2 z \,dt = \Big[\xi_2 + a_{21} y + a_{22} z + B_{21}\cdot\nabla y + B_{22}\cdot\nabla z + u_2 \chi_{G_1}\Big] \,dt + \Big[b_{21} y + b_{22} z  + u_4\Big] \,dW(t) & \textnormal{in } Q, \\[0.8em]
y = z = 0 & \textnormal{on } \Sigma, \\[0.8em]
(y(0,\cdot),z(0,\cdot)) = (\tau_1 \widehat{y}_0,\tau_2 \widehat{z}_0) & \textnormal{in } G,
\end{array}
\end{cases}
\end{equation}
where $(y,z)$ is the state variable, \( \xi_1,\xi_2 \in L^2_\mathcal{F}(0,T; L^2(G)) \) represent source terms. The coefficients \(a_{ij}\), \(b_{ij}\), and \(B_{ij}\) are assumed to satisfy
\[
a_{ij},\, b_{ij} \in L^\infty_{\mathcal{F}}(0,T; L^\infty(G)), 
\quad 
B_{ij}\in L^\infty_{\mathcal{F}}(0,T; L^\infty(G; \mathbb{R}^N)),
\quad \text{for } 1 \leq i,j \leq 2.
\] 
The control variable consists of four controls. The controls \(u_1\) and
\(u_2\) act, respectively, on the drift terms and are localized in the
spatial control regions \(G_0\) and \(G_1\), while the controls \(u_3\) and
\(u_4\) act throughout \(G\) in the diffusion terms, such that $(u_1,u_2,u_3,u_4) \in \mathcal{U}$. We identify \(u_i\), \(i=1,2\), with their extensions by zero to \(G\). The scalars \( \tau_1, \tau_2 \in \mathbb{R} \) are small unknown parameters. In this context, \( \tau_1 \widehat{y}_0 \) and \( \tau_2 \widehat{z}_0 \) represent small perturbations of the null initial data \( (0,0) \). We assume that \( \widehat{y}_0, \widehat{z}_0 \in L^2_{\mathcal{F}_0}(\Omega;L^2(G)) \) are unknown random variables satisfying the normalization condition
\begin{equation}\label{normmc}
\mathbb{E}\|(\widehat{y}_0,\widehat{z}_0)\|^2_{(L^2(G))^2} = 1.
\end{equation}

From a mathematical modeling standpoint, when $\xi_i = u_i = 0$
 and $\tau_1 = \tau_2 = 0$, system \eqref{ass15}
reduces to a diffusion model arising in stochastic reaction--diffusion--convection system. It describes the temporal evolution of the density of various quantities, including heat, bacterial populations, and chemical concentrations. The inclusion of stochastic noise driven by Brownian motion leads to a stochastic formulation that captures the cumulative effect of small, independent random perturbations acting during the evolution of the density. For further details on such models, we refer the reader to
\cite{DapratoZabcz} and \cite[Chapter~5]{lu2021mathematical},
and the references therein.

The main objective of this paper is to study the local insensitization of the following functional, referred to as a \emph{sentinel}. For each \(\beta \in [0,1]\), we define

\begin{align}\label{functioPhii1.2}
    \Phi^\beta_{\tau_1,\tau_2}(y, z) 
    := \frac{1-\beta}{2} \,\mathbb{E}\iint_Q \bigl(\chi_{\mathcal{O}_1}|z|^2 + \chi_{\mathcal{O}_2}|\nabla z|^2\bigr) \, dx \, dt
     + \frac{\beta}{2} \,\mathbb{E}\iint_Q \bigl(\chi_{\mathcal{O}_3}|y|^2 + \chi_{\mathcal{O}_4}|\nabla y|^2\bigr) \, dx \, dt,
\end{align}
where the pair \( (y, z) \) is the solution of the system \eqref{ass15} corresponding to the parameters \( \tau_1 \), \( \tau_2 \), and the controls \( (u_i)_{1\leq i\leq4} \). The insensitizing control problem, initially introduced by J.-L. Lions in \cite{lions1989quel}, consists of finding the controls \( u_i \) such that the functional \( \Phi^\beta_{\tau_1,\tau_2} \) remains locally invariant under small perturbations of the initial data \( (0,0) \). More precisely, the insensitizing control problem for the system \eqref{ass15} is formulated as follows: For any \( \xi_1, \xi_2 \in L^2_\mathcal{F}(0,T; L^2(G)) \), a quadruplet of control functions \( (u_i)_{1\leq i\leq4} \) is said to insensitize the functional \( \Phi^\beta_{\tau_1,\tau_2} \) if
\begin{align}\label{inspb}
    \frac{\partial \Phi^\beta_{\tau_1,\tau_2}(y, z)}{\partial \tau_1} \Bigg|_{\tau_1 = \tau_2 = 0} = \frac{\partial \Phi^\beta_{\tau_1,\tau_2}(y, z)}{\partial \tau_2} \Bigg|_{\tau_1 = \tau_2 = 0} = 0,
\end{align}
for all \( \widehat{y}_0, \widehat{z}_0 \in L^2_{\mathcal{F}_0}(\Omega;L^2(G)) \) satisfying the condition \eqref{normmc}.

The weighting parameter $\beta$ in the sentinel functional
\eqref{functioPhii1.2} determines the relative contribution of the two state
components to the insensitizing control problem. This leads to two distinct
regimes. When $\beta\in\{0,1\}$, the sentinel functional involves only one
state component and its spatial gradient. In this case, the observation
information is reduced, resulting in a more delicate coupling structure in
the associated adjoint system and, consequently, requiring a more involved
Carleman estimate, established in Theorem~\ref{carlthm32} for the case
$\beta=0$; the case $\beta=1$ can be treated by analogous arguments. In
contrast, when $\beta\in(0,1)$, both state components and their spatial
gradients enter the sentinel functional. The resulting observation structure
provides additional information and leads to a different Carleman estimate
for the associated adjoint system. Accordingly, we treat separately the two
regimes $\beta\in\{0,1\}$ and $\beta\in(0,1)$.

To ensure the solvability of the insensitizing control problems under
consideration, we impose the following assumptions on the control regions
\(G_i\), \(1\leq i\leq2\), the observation regions
\(\mathcal{O}_i\), \(1\leq i\leq4\), and the coupling coefficients
\(a_{21}\), \(b_{21}\), and \(B_{21}\).

\medskip
\noindent
\textbf{Assumptions.}
The following conditions are assumed to hold.

\begin{description}

\refstepcounter{assumption}
\item[\textnormal{\textbf{(\theassumption)}}]\label{assA1}
The control region \(G_0\) and the observation regions
\(\mathcal{O}_1,\mathcal{O}_2\) satisfy
\begin{equation*}
    G_0\cap\mathcal{O}_1\not\subseteq\mathcal{O}_2.
\end{equation*}

\refstepcounter{assumption}
\item[\textnormal{\textbf{(\theassumption)}}]\label{assA2}
The control regions \(G_0,G_1\) and the observation regions
\(\mathcal{O}_1,\ldots,\mathcal{O}_4\) satisfy
\begin{equation*}
    G_0\cap G_1\cap\mathcal{O}_1\cap\mathcal{O}_3
    \not\subseteq\mathcal{O}_2\cup\mathcal{O}_4.
\end{equation*}

\refstepcounter{assumption}
\item[\textnormal{\textbf{(\theassumption)}}]\label{assA3}
There exist a constant \(a_0>0\) and a nonempty open subset
\(\widetilde G_0\Subset G_0\cap\mathcal{O}_1\) such that
\begin{equation*}
    a_{21}\geq a_0
    \quad\text{or}\quad
    -a_{21}\geq a_0,
    \qquad\text{in }(0,T)\times\widetilde G_0,
    \quad\text{a.s.}
\end{equation*}

\refstepcounter{assumption}
\item[\textnormal{\textbf{(\theassumption)}}]\label{assA4}
The diffusion coupling coefficient \(b_{21}\) vanishes identically:
\begin{equation*}
    b_{21}=0,
    \qquad\text{in }Q,\quad\text{a.s.}
\end{equation*}

\refstepcounter{assumption}
\item[\textnormal{\textbf{(\theassumption)}}]\label{assA5}
The first-order coupling coefficient \(B_{21}\) vanishes identically:
\begin{equation*}
    B_{21}=0,
    \qquad\text{in }Q,\quad\text{a.s.}
\end{equation*}

\end{description}

System \eqref{ass15} is well posed. Indeed, it can be reformulated as the
following abstract forward system
\begin{equation}\label{esuq11.lk}
\begin{cases}
\begin{array}{ll}
d\mathbf{Y} - \mathcal{L}\mathbf{Y} \,dt = [\Xi + A \mathbf{Y}+B\cdot\nabla\textbf{Y} +  u] \,dt + [\widetilde{A}\textbf{Y}+\widetilde{u}] \,dW(t) & \textnormal{in } Q, \\
\mathbf{Y} = 0 & \textnormal{on } \Sigma, \\
\mathbf{Y}(0) = \mathbf{Y}_0 & \textnormal{in } G,
\end{array}
\end{cases}
\end{equation}
where \( \mathbf{Y} = (y, z)^\top \) is the state variable, \( \nabla \mathbf{Y} = (\nabla y, \nabla z)^\top \), \( \mathbf{Y}_0 =(\tau_1 \widehat{y}_0,\tau_2 \widehat{z}_0)^\top\in L^2_{\mathcal{F}_0}(\Omega;L^2(G; \mathbb{R}^2)) \) is the initial state, \( \Xi = (\xi_1, \xi_2)^\top \) is the source term, \( u = (u_1\chi_{G_0}, u_2\chi_{G_1})^\top \) is the control in the drift, \( \widetilde{u}  = (u_3, u_4)^\top \) represents the additional control in the diffusion term, and the coupling matrices

\[ A = (a_{ij})_{1 \leq i, j \leq 2}, \;\,\widetilde{A} = (b_{ij})_{1 \leq i, j \leq 2} \in L_\mathcal{F}^\infty(0, T; L^\infty(G; \mathbb{R}^{2 \times 2})), \]
\[ B = (B_{ij})_{1 \leq i,j \leq 2}  \in L_\mathcal{F}^\infty(0, T; L^\infty(G; \mathbb{R}^{2N \times 2})).\]  
The operator \( \mathcal{L} = \text{diag}(\mathcal{L}_1, \mathcal{L}_2) \), where \( \mathcal{L}_m \) (\( m = 1,2 \)) are the self-adjoint second-order operators defined in \eqref{definofopeLi}. By standard well-posedness results for linear stochastic parabolic equations
(see, e.g., \cite{krylov1994}), for any \( \Xi \in L^2_\mathcal{F}(0,T; L^2(G; \mathbb{R}^2)) \), \( \mathbf{Y}_0 \in L^2_{\mathcal{F}_0}(\Omega;L^2(G; \mathbb{R}^2)) \), and \( (u,\widetilde{u}) \in \mathcal{U} \), the system \eqref{esuq11.lk} admits a unique weak solution
\[
\mathbf{Y} = (y, z)^\top \in L^2_\mathcal{F}(\Omega; C([0,T]; L^2(G; \mathbb{R}^2))) \bigcap L^2_\mathcal{F}(0,T; H^1_0(G; \mathbb{R}^2)).
\]
Moreover, there exists a constant \( C > 0 \) such that
\[
\begin{aligned}
    & \| \mathbf{Y} \|_{L^2_\mathcal{F}(\Omega; C([0,T]; L^2(G; \mathbb{R}^2)))} + \| \mathbf{Y} \|_{L^2_\mathcal{F}(0,T; H^1_0(G; \mathbb{R}^2))} \\
    & \leq C \Big( \| \mathbf{Y}_0 \|_{L^2_{\mathcal{F}_0}(\Omega;L^2(G; \mathbb{R}^2))} + \|\Xi\|_{L^2_\mathcal{F}(0,T; L^2(G; \mathbb{R}^2))}+ \|(u,\widetilde{u})\|_{\mathcal{U}} \Big).
\end{aligned}
\]

The main result of this paper concerns the existence of insensitizing controls for
system~\eqref{ass15} in the sense of~\eqref{inspb}. Depending on the value of the
parameter $\beta$, we distinguish between the following two cases:
\begin{enumerate}[label=(\roman*)]
\item\label{case:i} $\beta = 0$ or $\beta = 1$;
\item\label{case:ii} $\beta \in (0,1)$.
\end{enumerate}
For case~\ref{case:i}, we present the analysis for $\beta=0$ only; by the symmetry
between the variables $y$ and $z$, the case $\beta=1$ can be handled by analogous
arguments and is therefore omitted. We now state the first main result of the paper.
\begin{thm}[Case $\beta = 0$]\label{thmm1.3ins}
Assume that \eqref{assA1}, \eqref{assA3}, \eqref{assA4}, and \eqref{assA5} hold. Then there exist constants $M>0$ and $C>0$, depending only on
$G$, $G_0$, $\mathcal{O}_1$, $\mathcal{O}_2$, $T$, $a_0$, $a_{ij}$, $b_{ij}$,
and $B_{ij}$, such that for any $\xi_1,\xi_2\in L^2_{\mathcal F}(0,T;L^2(G))$
satisfying
\begin{equation}\label{assonx1xi2ins}
\mathbb{E}\iint_Q \exp(Mt^{-1})\bigl(|\xi_1|^2+|\xi_2|^2\bigr)\,dx\,dt<\infty,
\end{equation}
there exists a control triple $(u_1,u_3,u_4)\in\mathcal{U}_0$, with $u_2\equiv0$,
that insensitizes the functional $\Phi^0_{\tau_1,\tau_2}$ in the sense
of \eqref{inspb}. Moreover, these controls satisfy the estimate
\begin{equation}\label{estthmm1.3ins}
\|(u_1,u_3,u_4)\|^2_{\mathcal{U}_0}
\leq C\,\mathbb{E}\iint_Q \exp(Mt^{-1})\bigl(|\xi_1|^2+|\xi_2|^2\bigr)\,dx\,dt.
\end{equation}
\end{thm}

We next state our second main result.

\begin{thm}[Case $\beta\in(0,1)$]\label{thmm1.3ins3}
Assume that \eqref{assA2} holds. Then there exist constants $M>0$ and $C>0$,
depending only on $G$, $G_i$, $\mathcal{O}_i$, $T$, $a_{ij}$, $b_{ij}$, and
$B_{ij}$, such that for any $\xi_1,\xi_2\in L^2_{\mathcal F}(0,T;L^2(G))$
satisfying
\begin{equation}\label{assonx1xi2ins3}
\mathbb{E}\iint_Q \exp(Mt^{-1})\bigl(|\xi_1|^2+|\xi_2|^2\bigr)\,dx\,dt<\infty,
\end{equation}
there exist controls $(u_1,u_2,u_3,u_4)\in\mathcal{U}$ that insensitize the
functional $\Phi^\beta_{\tau_1,\tau_2}$ in the sense of \eqref{inspb}.
Moreover, these controls satisfy the estimate
\begin{equation}\label{estthmm1.3ins3}
\|(u_1,u_2,u_3,u_4)\|^2_{\mathcal{U}}
\leq C\,\mathbb{E}\iint_Q \exp(Mt^{-1})\bigl(|\xi_1|^2+|\xi_2|^2\bigr)\,dx\,dt.
\end{equation}
\end{thm}

Insensitizing control problems aim to design controls that make a prescribed functional of the system state insensitive, at first order, to small perturbations of uncertain data, such as the initial or boundary conditions. Since their introduction by J.-L. Lions \cite{lions1989quel}, these problems
have been extensively studied in the deterministic setting and have become an important class of control problems for partial differential equations.

In the deterministic framework, insensitizing controls have been investigated for a wide variety of parabolic equations with static boundary conditions (see, e.g., \cite{bodafabre95,BodarBurgosPerez2004,BodarBurgosPerez04NonAnalysis,BodaGnPer,Tereza97Esaim}) and, more recently, for equations with dynamic boundary conditions \cite{bouetman25,zhanyingaodbc19}. These studies have also shown that the existence of insensitizing controls may depend essentially on the choice of the initial data; see, for instance, \cite{Tereza2000,Tereidenfication}. Extensions to sentinel functionals involving spatial gradients of the state have been considered in \cite{gureSiam07,msacar26}. The theory has subsequently been extended to various nonlinear, higher-order, dispersive, and coupled systems, including the nonlinear Schrödinger equation \cite{capfiltanka20}, the Hirota--Satsuma system \cite{bhann24}, fourth-order parabolic equations \cite{kass20}, the Ginzburg--Landau equation \cite{santa19}, the Kawahara equation \cite{kumamaj25}, and several coupled deterministic systems \cite{ghadsanta24,calcarcerpa16,Djomkenn25}. Numerical aspects have also been investigated for semilinear parabolic equations; see \cite{boyhater19}.

The stochastic theory of insensitizing controls is considerably less
developed. To the best of our knowledge, existing results have mainly
concerned single stochastic equations, including backward stochastic heat
equations \cite{liu2014global}, stochastic degenerate parabolic equations
\cite{liu2019carleman}, forward stochastic heat equations
\cite{yansun2011}, the stochastic Kuramoto--Sivashinsky equation
\cite{luliu25}, and stochastic parabolic equations with dynamic boundary
conditions \cite{elgrou1D23insensi24}. At the same time, the controllability
theory for stochastic parabolic equations has undergone substantial
development, including results for coupled systems
\cite{elgconvec23,Preprielgr24jmaa,
liu14couplfor,LiuuLiuX,liu2019carleman}; see also the monograph
\cite{lu2021mathematical}. These developments provide important tools for
the study of more general stochastic control problems, including
multi-objective problems in the Stackelberg--Nash framework
\cite{elgomar26,omboelman25sec,yuzhang23t}.

Despite these developments, several important gaps remain in the theory of
insensitizing controls. First, the stochastic theory is still relatively
limited, with most existing results concerning single stochastic equations.
Second, insensitizing control problems for genuinely coupled stochastic
parabolic systems remain largely undeveloped, despite the additional
difficulties caused by the coupling terms in the associated adjoint systems.
Third, the inclusion of spatial gradient observations introduces further
analytical difficulties and generally requires stronger geometric conditions
on the control and observation regions. Even in the deterministic parabolic setting, gradient observations require
more delicate Carleman estimates and stronger geometric conditions than state
observations alone; see Remark~\ref{rmk1.1h} for further details.

The present paper addresses these three issues simultaneously. We study
insensitizing control problems for a class of coupled linear stochastic
parabolic systems governed by two forward stochastic
reaction--convection--diffusion equations, with sentinel functionals
involving both localized observations of the state variables and their
spatial gradients. The associated adjoint system contains zeroth- and
second-order coupling terms generated by the observation operator. Our main
analytical contribution is the derivation of new global Carleman estimates
for such coupled stochastic parabolic systems. These estimates yield the
corresponding observability inequalities and, through a duality argument,
the existence of insensitizing controls.

Carleman estimates constitute one of the fundamental analytical tools in the
study of control problems for parabolic equations, with applications to uniqueness,
observability, inverse problems, controllability, and insensitizing control.
They were introduced by Carleman \cite{Carl39} in the study of uniqueness
for second-order elliptic equations and have since been developed extensively
in both deterministic and stochastic settings. For general references on
Carleman estimates and their applications, we refer the reader to
\cite{surveyAmmarKBGT,coron07,fernandez2006global,fursikov1996controllability,
luZhang22mcrf,tang2009null}.

In Theorem~\ref{thmm1.3ins}, assumption \eqref{assA3} on the coupling coefficient \(a_{21}\) allows us to achieve insensitization with only three controls, namely \(u_1\), \(u_3\), and \(u_4\). Indeed, the term
\(a_{21}y\) provides an indirect coupling mechanism through which the control
\(u_1\), acting on the first equation, influences the second state component.
In contrast, eliminating the additional controls \(u_3\) and \(u_4\), which
act in the diffusion terms, remains a challenging open problem. This
difficulty is already present in the classical null controllability theory
for forward stochastic parabolic equations; see, e.g.,
\cite{elgconvec23,tang2009null}. For related results on partial
controllability with a single control for stochastic heat equations, we
refer to \cite{lu2011some,observineqback}. Controllability
results for coupled stochastic parabolic systems with two equations can be
found in \cite{liu14couplfor,LiuuLiuX}. Finally, assumptions
\eqref{assA4} and \eqref{assA5} play an essential role in the derivation of
the Carleman estimate in Theorem~\ref{carlthm32}; see
Remark~\ref{rkk15.5} for further discussion.

Now, some remarks are in order.
\begin{rmk}
The assumption that the coefficients $\sigma_{ij}^m$ have $W^{2,\infty}$-regularity with respect to the spatial variable may appear rather strong. In \cite{tang2009null}, Carleman estimates for forward and backward stochastic parabolic equations were established using an exponential weighted energy identity (see Theorem 3.1), under the assumption of $W^{2,\infty}$-regularity of the coefficients $b^{ij}$. On the other hand, in \cite{liu2014global}, it was shown, by means of a duality argument, that $W^{1,\infty}$-regularity is sufficient for Carleman estimates for forward stochastic parabolic equations (see Remark 1.2). To the best of our knowledge, whether the $W^{2,\infty}$-regularity assumption can be weakened to $W^{1,\infty}$ for backward stochastic parabolic equations remains an unresolved problem. Since our analysis requires Carleman estimates for both forward and backward equations, we assume that the coefficients $\sigma_{ij}^m$ have $W^{2,\infty}$-regularity with respect to the spatial variable.
\end{rmk}

\begin{rmk}\label{rmk1.1h}
In this work, we investigate insensitizing control problems associated with
the functional \eqref{functioPhii1.2}, which involves both state and
gradient observations. In the absence of gradient observations (i.e., when
$\mathcal{O}_2=\mathcal{O}_4=\emptyset$), assumptions \eqref{assA1} and \eqref{assA2} reduce to
\[
G_0\cap\mathcal{O}_1\neq\emptyset
\quad\text{and}\quad
G_0\cap G_1\cap\mathcal{O}_1\cap\mathcal{O}_3\neq\emptyset,
\]
respectively. These are the standard geometric conditions in the literature
on insensitizing controls with state observations. The presence of gradient
observations, however, requires the stronger geometric assumptions
\eqref{assA1} and \eqref{assA2}, which are crucial for establishing
the main Carleman estimates in Theorems~\ref{carlthm32} and
\ref{carlthm3233.4}, respectively. Whether these additional geometric
assumptions can be removed remains an open problem, as doing so would
entail substantial technical difficulties; see \cite{gureSiam07} for the
corresponding deterministic parabolic setting.
\end{rmk}

\begin{rmk}\label{rkk15.5}
The assumptions \eqref{assA4} and \eqref{assA5} play an essential role in
the derivation of the Carleman estimate \eqref{carlestcasca} for the case
$\beta=0$. Indeed, these assumptions allow us to eliminate several
problematic terms arising in the proof, which cannot otherwise be absorbed
into the left-hand side by means of Young's inequality. More precisely, the
problematic terms arising in Steps~2, 3, and~6 of the proof are

\begin{align*} &\text{Step 2:}\qquad \mathbb{E}\iint_Q w_3\zeta_1 q\, B_{21}\cdot\nabla q\,dx\,dt,\qquad \mathbb{E}\iint_Q w_3\zeta_1 b_{11}b_{21} q^2 \,dx\,dt, \\ &\text{Step 3:}\qquad \mathbb{E}\iint_Q w_5\zeta_2 b_{21}^2 q^2 \,dx\,dt, \\ &\text{Step 6:}\qquad \mathbb{E}\iint_Q k\, B_{21}\cdot\nabla(w_{44}\zeta_5 k)\,dx\,dt. \end{align*}
Furthermore, the method developed in this paper does not extend to the inclusion of first-order terms in the diffusion parts of system \eqref{ass15}. The main obstacle appears when applying the Carleman estimate \eqref{carfor5.6new} to the corresponding extended equation associated with \eqref{adjoforr4.8}. In this situation, one would need to absorb a term of the form
\[ \lambda^{d-1}\mathbb{E}\iint_Q \theta^2\gamma^{d-1}|\nabla z|^2\,dx\,dt \] 
on the right-hand side by the left-hand side term
\[ \lambda^{d-2}\mathbb{E}\iint_Q \theta^2\gamma^{d-2}|\nabla z|^2\,dx\,dt, \]
which is impossible since $d-2<d-1$. This explains the restrictions imposed in this paper on the diffusion terms in \eqref{ass15}.
\end{rmk}

\begin{rmk}
Natural extensions of the sentinel functional \eqref{functioPhii1.2} can be obtained by introducing coupling terms between the two state components. For instance, let $\mathcal{O}_5,\mathcal{O}_6\subset G$ be nonempty open subsets and consider
the following additional terms in \eqref{functioPhii1.2}:
\[
\gamma\chi_{\mathcal O_5}yz
+\eta\chi_{\mathcal O_6}\nabla y\cdot\nabla z,
\]
or
\[
\gamma\chi_{\mathcal O_5}|y-z|^2
+\eta\chi_{\mathcal O_6}|\nabla y-\nabla z|^2,
\qquad \gamma,\eta\in\mathbb{R}.
\]
The first type of terms measures the local interaction between the two
components, while the second measures their local mismatch. Such
functionals are relevant when the relative behavior of the coupled states
is of interest. From an analytical viewpoint, these terms introduce
additional zeroth- and second-order cross-couplings into the associated
adjoint system, leading to non-diagonal observation operators and more
delicate Carleman estimates. The corresponding insensitizing control
problems are left for future investigation.
\end{rmk}

The paper is organized as follows: In Section~\ref{sec2}, we reformulate the insensitizing control problem~\eqref{inspb} as a null controllability problem for a cascade system of coupled forward--backward stochastic parabolic equations. Section~\ref{sec3} is devoted to the derivation of the main Carleman estimates for the associated adjoint backward--forward stochastic parabolic system, distinguishing between the cases $\beta=0$ and $\beta\in(0,1)$. In Section~\ref{sec4}, we establish the corresponding observability inequalities. Finally, Section~\ref{sec5} presents the proof of the main insensitizing control results, namely Theorems~\ref{thmm1.3ins} and~\ref{thmm1.3ins3}.

\section{From Insensitizing Control to Null Controllability}\label{sec2}
In this section, we reformulate the insensitizing control problem~\eqref{inspb}
for any $\beta \in[0,1]$ as a null controllability problem for a suitable  coupled cascade 
forward--backward stochastic parabolic system.

\begin{prop}\label{proposs11}
Let $\xi_1, \xi_2 \in L^2_\mathcal{F}(0,T;L^2(G))$ and let $((y, z); (p, r; P, R))$ be the solution of the following systems of linear forward-backward stochastic parabolic equations associated with the control quadruple  $(u_1, u_2, u_3,u_4) \in \mathcal{U}$:
\begin{equation}\label{forr4.1}
\begin{cases}
\begin{array}{ll}
dy - \mathcal{L}_1 y \, dt = \Big[\xi_1 + a_{11} y + a_{12} z + B_{11}\cdot\nabla y + B_{12}\cdot\nabla z + u_1 \chi_{G_0}\Big] \, dt\\
\hspace{2.5cm}+ \Big[ b_{11} y + b_{12} z +u_3\Big] \, dW(t) & \textnormal{in } Q, \\[0.8em]
dz - \mathcal{L}_2 z \, dt = \Big[\xi_2 + a_{21} y + a_{22} z+ B_{21}\cdot\nabla y+ B_{22}\cdot\nabla z + u_2 \chi_{G_1}\Big] \, dt\\
\hspace{2.5cm}+ \Big[b_{21} y + b_{22} z+u_4\Big] \, dW(t) & \textnormal{in } Q, \\[0.8em]
y = z = 0 & \textnormal{on } \Sigma, \\[0.8em]
y(0,\cdot) = z(0,\cdot) = 0& \textnormal{in } G,
\end{array}
\end{cases}
\end{equation}
and 
\begin{equation}\label{back45}
\begin{cases}
\begin{array}{ll}
dp + \mathcal{L}_1 p \, dt = \Big[-a_{11} p - a_{21} r-b_{11} P - b_{21} R +\nabla\cdot(pB_{11}+rB_{21})\\
\hspace{2,5cm}-\beta( \chi_{\mathcal{O}_3} y-\nabla\cdot(\chi_{\mathcal{O}_4}\nabla y))\Big] \, dt + P \, dW(t) & \textnormal{in } Q, \\[0.8em]
dr + \mathcal{L}_2 r \, dt = \Big[-a_{12} p - a_{22} r-b_{12} P - b_{22} R+\nabla\cdot(pB_{12}+rB_{22}) \\
\hspace{2.5cm}- (1-\beta)(\chi_{\mathcal{O}_1} z-\nabla\cdot(\chi_{\mathcal{O}_2}\nabla z))\Big] \, dt + R \, dW(t) & \textnormal{in } Q, \\[0.8em]
p = r = 0 & \textnormal{on } \Sigma, \\[0.8em]
p(T,\cdot) = r(T,\cdot) = 0 & \textnormal{in } G.
\end{array}
\end{cases}
\end{equation}
Then, the insensitizing control problem \eqref{inspb} holds for the controls $(u_1, u_2, u_3,u_4)$ if and only if the solution $((y, z); (p, r; P, R))$ of systems \eqref{forr4.1}–\eqref{back45} satisfies that
\begin{align}\label{nullcontroprop}
p(0,\cdot) = r(0,\cdot) = 0 \quad \textnormal{in } G, \quad \textnormal{a.s.}
\end{align}
\end{prop}
\begin{proof} 
We denote by $(y_\tau, z_\tau)$ the solution of the system \eqref{ass15} associated with the parameters $\tau_1, \tau_2$ and the controls $u_1, u_2, u_3, u_4$. Let $\widehat{y}_0, \widehat{z}_0 \in L^2_{\mathcal{F}_0}(\Omega;L^2(G))$ be the unknown initial perturbations of the null initial datum, respectively, such that $\|(\widehat{y}_0,\widehat{z}_0)\|_{(L^2_{\mathcal{F}_0}(\Omega;L^2(G)))^2} = 1$. By a direct computation of the partial derivative of $\Phi^\beta_{\tau_1,\tau_2}$ with respect to the parameter $\tau_1$, it is straightforward to see that
\begin{align}\label{partdert43}
\begin{aligned}
\frac{\partial \Phi_{\tau_1, \tau_2}(y_\tau, z_\tau)}{\partial \tau_1} \Bigg|_{\tau_1 = \tau_2 = 0} = \lim_{\tau_1 \to 0} \Bigg[&\frac{1-\beta}{2} \mathbb{E} \int_0^T \int_{\mathcal{O}_1} (z_{\tau_1} + z) \frac{z_{\tau_1} - z}{\tau_1} \, dx \, dt\\
&+\frac{1-\beta}{2} \mathbb{E} \int_0^T \int_{\mathcal{O}_2} \nabla(z_{\tau_1} + z) \cdot\nabla\bigg(\frac{z_{\tau_1} - z}{\tau_1}\bigg) \, dx \, dt\\
&+\frac{\beta }{2} \mathbb{E} \int_0^T \int_{\mathcal{O}_3} (y_{\tau_1} + y) \frac{y_{\tau_1} - y}{\tau_1} \, dx \, dt\\
&+\frac{\beta }{2} \mathbb{E} \int_0^T \int_{\mathcal{O}_4} \nabla(y_{\tau_1} + y) \cdot\nabla\bigg(\frac{y_{\tau_1} - y}{\tau_1}\bigg) \, dx \, dt\Bigg],
\end{aligned}
\end{align}
where $(y, z)$ is the solution of \eqref{forr4.1}, and $(y_{\tau_1}, z_{\tau_1})$ is the solution of the following system
\begin{equation}\label{perturbed_system}
\begin{cases}
\begin{array}{ll}
dy_{\tau_1} - \mathcal{L}_1 y_{\tau_1} \,dt = \Big[\xi_1 + a_{11} y_{\tau_1} + a_{12} z_{\tau_1}+B_{11}\cdot\nabla y_{\tau_1} + B_{12}\cdot\nabla z_{\tau_1}  +  u_1 \chi_{G_0}\Big] \,dt \\
\hspace{2.7cm}+ \Big[ b_{11} y_{\tau_1} + b_{12} z_{\tau_1} + u_3\Big] \, dW(t) & \textnormal{in } Q, \\[0.8em]
dz_{\tau_1} - \mathcal{L}_2 z_{\tau_1} \,dt = \Big[\xi_2 + a_{21} y_{\tau_1} + a_{22} z_{\tau_1}+B_{21}\cdot\nabla y_{\tau_1}+B_{22}\cdot\nabla z_{\tau_1} +  u_2 \chi_{G_1}\Big] \,dt 
\\
\hspace{2.7cm}+ \Big[b_{21} y_{\tau_1} + b_{22} z_{\tau_1}+u_4\Big] \, dW(t) & \textnormal{in } Q, \\[0.8em]
y_{\tau_1} = z_{\tau_1} = 0 & \textnormal{on } \Sigma, \\[0.8em]
y_{\tau_1}(0,\cdot) =  \tau_1 \widehat{y}_0, \quad z_{\tau_1}(0,\cdot) = 0 & \textnormal{in } G.
\end{array}
\end{cases}
\end{equation}
Put
\[
\overline{y}=\frac{y_{\tau_1}-y}{\tau_1}
\quad\text{and}\quad
\overline{z}=\frac{z_{\tau_1}-z}{\tau_1}.
\]
Then it is straightforward to verify that \((\overline{y},\overline{z})\) is the solution of the following system
\begin{equation}\label{forr4.4}
\begin{cases}
\begin{array}{ll}
d\overline{y} - \mathcal{L}_1 \overline{y} \, dt = \Big[a_{11} \overline{y} + a_{12} \overline{z}+B_{11} \cdot\nabla\overline{y} + B_{12} \cdot\nabla\overline{z}\Big] \, dt+\Big[b_{11} \overline{y} + b_{12} \overline{z}\Big] \, dW(t)  & \textnormal{in } Q, \\[0.8em]
d\overline{z} - \mathcal{L}_2 \overline{z} \, dt = \Big[a_{21} \overline{y} + a_{22} \overline{z}+ B_{21} \cdot\nabla\overline{y}+ B_{22} \cdot\nabla\overline{z}\Big] \, dt+\Big[b_{21} \overline{y} + b_{22} \overline{z}\Big] \, dW(t)  & \textnormal{in } Q, \\[0.8em]
\overline{y} = \overline{z} = 0 & \textnormal{on } \Sigma, \\[0.8em]
\overline{y}(0,\cdot) = \widehat{y}_0, \quad \overline{z}(0,\cdot) = 0 & \textnormal{in } G.
\end{array}
\end{cases}
\end{equation}
We note that the solution of \eqref{forr4.4} is independent of \(\tau_1\). Consequently, by \eqref{partdert43}, we obtain
\begin{align}\label{equdepar36}
\begin{aligned}
 \frac{\partial \Phi^\beta_{\tau_1,\tau_2}(y_\tau,z_\tau)}{\partial\tau_1} \Bigg|_{\tau_1 = \tau_2 = 0} = 
 &\,(1-\beta)\left(\mathbb{E} \int_0^T \int_{\mathcal{O}_1} z \overline{z} \, dx \, dt+\mathbb{E} \int_0^T \int_{\mathcal{O}_2} \nabla z \cdot\nabla \overline{z} \, dx \, dt\right)\\
 &+\beta\left(\mathbb{E} \int_0^T \int_{\mathcal{O}_3} y \overline{y} \, dx \, dt+\mathbb{E} \int_0^T \int_{\mathcal{O}_4} \nabla y \cdot\nabla \overline{y} \, dx \, dt\right).
 \end{aligned}
\end{align}
Applying Itô's formula to the systems \eqref{back45} and \eqref{forr4.4}, we deduce
\begin{align}\label{eqq2,78}
\begin{aligned}
&\,(1-\beta)\left(\mathbb{E} \int_0^T \int_{\mathcal{O}_1} z \overline{z} \, dx \, dt+\mathbb{E} \int_0^T \int_{\mathcal{O}_2} \nabla z \cdot\nabla \overline{z} \, dx \, dt\right)\\
 &+\beta\left(\mathbb{E} \int_0^T \int_{\mathcal{O}_3} y \overline{y} \, dx \, dt+\mathbb{E} \int_0^T \int_{\mathcal{O}_4} \nabla y \cdot\nabla \overline{y} \, dx \, dt\right)\\
&= \mathbb{E}\int_G \widehat{y}_0 p(0) \, dx.
\end{aligned}
\end{align}
Combining \eqref{eqq2,78} and \eqref{equdepar36}, we end up with
\begin{align}\label{partdert432.8}
\frac{\partial \Phi^\beta_{\tau_1,\tau_2}(y_\tau, z_\tau)}{\partial \tau_1} \Bigg|_{\tau_1 = \tau_2 = 0} = \mathbb{E}\int_G \widehat{y}_0 p(0) \, dx.
\end{align}
Proceeding as in the above computations, we derive that
\begin{align}\label{partdert43tau2}
        \frac{\partial \Phi^\beta_{\tau_1,\tau_2}(y_\tau,z_\tau)}{\partial\tau_2}\Bigg|_{\tau_1=\tau_2=0} = \mathbb{E}\int_G \widehat{z}_0 r(0) \,dx.
\end{align}
From \eqref{partdert432.8} and \eqref{partdert43tau2}, we conclude that the insensitivity control problem \eqref{inspb} is satisfied if and only if
\[
 \mathbb{E}\int_G \widehat{y}_0 p(0) \, dx = \mathbb{E} \int_G \widehat{z}_0 r(0) \, dx = 0,
\]
for all \( \widehat{y}_0, \widehat{z}_0 \in L^2_{\mathcal{F}_0}(\Omega;L^2(G)) \) such that \( \|(\widehat{y}_0,\widehat{z}_0)\|_{(L^2_{\mathcal{F}_0}(\Omega;L^2(G)))^2} = 1 \). Hence, we conclude that the insensitivity control problem \eqref{inspb} is equivalent to the null controllability problem \eqref{nullcontroprop}. This completes the proof of Proposition \ref{proposs11}.
\end{proof}

\begin{rmk}
The systems \eqref{forr4.1}--\eqref{back45} are coupled in a cascade manner. The controls \(u_i\), $1\leq i\leq4$, act only on the forward system
\eqref{forr4.1}, while the backward system \eqref{back45} is controlled
indirectly by the forward states \(y\) and \(z\) through the zeroth- and
second-order observation terms
\[
\beta\bigl(\chi_{\mathcal O_3}y
-\nabla\cdot(\chi_{\mathcal O_4}\nabla y)\bigr),
\qquad
(1-\beta)\bigl(\chi_{\mathcal O_1}z
-\nabla\cdot(\chi_{\mathcal O_2}\nabla z)\bigr).
\]
When \(\beta=0\), the backward system is driven only through the state
\(z\), allowing the control of \(r\), which is then transferred to \(p\)
through the coupling term \(a_{21}r\). When \(\beta\in(0,1)\), both states
\(y\) and \(z\) contribute to the backward system. Thus, the
forward--backward cascade interaction provides the mechanism leading to
null controllability, namely, the condition \eqref{nullcontroprop}.
\end{rmk}

Through the classical duality argument, the null controllability condition \eqref{nullcontroprop} for the system of equations \eqref{forr4.1}--\eqref{back45} can be equivalently reformulated as an observability inequality for the following coupled backward-forward stochastic parabolic systems
\begin{equation}\label{adback4.77}
\begin{cases}
\begin{array}{ll}
dh + \mathcal{L}_1 h \, dt = \Big[-a_{11} h - a_{21} k-b_{11} H - b_{21} K+\nabla\cdot(hB_{11}+kB_{21})\\
\hspace{2.5cm}- \beta(\chi_{\mathcal{O}_3} q-\nabla\cdot(\chi_{\mathcal{O}_4}\nabla q))\Big] \, dt + H \, dW(t), & \text{in } Q, \\[0.8em]
dk + \mathcal{L}_2 k \, dt = \Big[-a_{12} h - a_{22} k-b_{12} H - b_{22} K +\nabla\cdot(hB_{12}+kB_{22}) \\
\hspace{2.5cm}- (1-\beta)(\chi_{\mathcal{O}_1} v-\nabla\cdot(\chi_{\mathcal{O}_2} \nabla v))\Big] \, dt + K \, dW(t), & \text{in } Q, \\[0.8em]
h = k = 0, & \text{on } \Sigma, \\[0.8em]
h(T,\cdot) = k(T,\cdot) = 0, & \text{in } G,
\end{array}
\end{cases}
\end{equation}
and
\begin{equation}\label{adjoforr4.8}
\begin{cases}
\begin{array}{ll}
dq - \mathcal{L}_1 q \, dt = \Big[a_{11} q + a_{12} v+B_{11}\cdot\nabla q + B_{12}\cdot\nabla v\Big] \, dt+\Big[b_{11} q + b_{12} v\Big] \, dW(t) , & \text{in } Q, \\[0.8em]
dv - \mathcal{L}_2 v \, dt = \Big[a_{21} q + a_{22} v+ B_{21}\cdot\nabla q+ B_{22}\cdot\nabla v\Big] \, dt+\Big[b_{21} q + b_{22} v\Big] \, dW(t) , & \text{in } Q, \\[0.8em]
q = v = 0, & \text{on } \Sigma, \\[0.8em]
q(0,\cdot) = q_0, \quad v(0,\cdot) = v_0, & \text{in } G,
\end{array}
\end{cases}
\end{equation}
where \( q_0, v_0 \in L^2_{\mathcal{F}_0}(\Omega;L^2(G)) \).

\section{Global Carleman Estimates with Different Values of $\beta$}\label{sec3}

This section is devoted to the derivation of global Carleman estimates for the
coupled backward--forward system~\eqref{adback4.77}--\eqref{adjoforr4.8}, according
to the following two cases for the parameter $\beta$:
\begin{enumerate}[i)]
    \item $\beta = 0$,
    \item $\beta \in(0,1)$.
\end{enumerate}

\subsection{Weight Functions and Preliminary Carleman Estimates}
In this subsection, we provide some essential tools for Carleman estimates for forward and backward stochastic parabolic equations. First, we recall the following known result (see \cite{fursikov1996controllability}).
\begin{lm}\label{lmm5.1}
For any nonempty open subset $\mathcal{B}\Subset G$, there exists a function $\psi\in C^4(\overline{G})$ such that
$$
\psi>0\;\,\, \textnormal{in} \,\,G\,;\qquad \psi=0\;\,\,\, \textnormal{on} \,\,\Gamma;\qquad\vert\nabla\psi\vert>0\; \,\,\,\,\textnormal{in}\,\,\overline{G\setminus\mathcal{B}}.
$$
\end{lm}
For any parameters $\lambda, \mu\geq1$, we introduce the following
Carleman weight functions:
\begin{equation}\label{2.2012}
\begin{cases}
\displaystyle \gamma\equiv\gamma(t)=[t(T-t)]^{-1},\\[0.4em]
\displaystyle \alpha\equiv\alpha(t,x)
=\bigl(\exp(\mu\psi(x))-\exp(2\mu\|\psi\|_{\infty})\bigr)\gamma(t),\\[0.4em]
\displaystyle \theta\equiv\theta(t,x)=\exp\bigl(\lambda\alpha(t,x)\bigr).
\end{cases}
\end{equation}

It is easy to check that there exists a constant $C=C(G)>0$ such that for all $(t,x)\in Q$ and $s>0$, 
\begin{equation}\label{2.301}
\begin{array}{l}
\gamma^s(t)\geq CT^{-2s},\qquad\vert\gamma'(t)\vert\leq CT\gamma^2(t),\qquad\vert\gamma''(t)\vert\leq CT^2\gamma^3(t),\\\\
\vert\alpha_t(t,x)\vert\leq CT\exp(2\mu||\psi||_\infty)\gamma^2(t),\qquad\vert\alpha_{tt}(t,x)\vert\leq CT^2\exp(2\mu||\psi||_\infty)\gamma^3(t).
		\end{array}
	\end{equation}
In what follows, for any $d \in \mathbb{R}$, we denote by
\[
\mathcal{I}(d, \cdot) = \mathbb{E} \iint_Q \Big\{ \lambda^d \theta^2 \gamma^d |\cdot|^2 + \lambda^{d-2}  \theta^2 \gamma^{d-2} \vert\nabla \cdot\vert^2\Big\} \, dx \, dt.
\]

Let us first consider the following forward stochastic parabolic equation
\begin{equation}\label{eqqgfr}
\begin{cases}
\begin{array}{ll}
dz - \displaystyle\sum_{i,j=1}^N \frac{\partial}{\partial x_i}\left(\sigma_{ij}^0(\omega,t,x)\frac{\partial z}{\partial x_j}\right) \,dt = F_0\,dt + F_1\,dW(t)&\textnormal{in}\,\,Q,\\
z=0&\textnormal{on}\,\,\Sigma,\\
z(0,\cdot)=z_0&\textnormal{in}\,\,G,
\end{array}
\end{cases}
\end{equation}
where $z_0\in L^2_{\mathcal{F}_0}(\Omega;L^2(G))$, $F_0, F_1\in L^2_\mathcal{F}(0,T;L^2(G))$, and the coefficients $\sigma_{ij}^0$ satisfy
\begin{enumerate}[(H1)]
\item $\sigma_{ij}^0\in L^\infty_\mathcal{F}(\Omega;C^1([0,T];W^{1,\infty}(G)))$.
\item $\sigma_{ij}^0=\sigma_{ji}^0$ for any $1\leq i,j\leq N$.
\item There exists a positive constant $\sigma^0$ so that 
\begin{align}\label{asspforcarestin}
\sum_{i,j=1}^N \sigma_{ij}^0(\omega,t,x)\kappa_i\kappa_j\geq \sigma^0|\kappa|^2\qquad\textnormal{for any}\quad (\omega,t,x,\kappa)\in \Omega\times Q\times\mathbb{R}^N.
\end{align}
\end{enumerate}

By applying the Carleman estimate  \cite[Theorem 1.1]{liu2014global} to the equation satisfied by the weighted state ``$(\lambda\gamma)^{\frac{d-3}{2}}z$'', and taking $\lambda$ sufficiently large, it is easy to derive the following general global Carleman estimate for system~\eqref{eqqgfr}.

\begin{lm}\label{thm3.3}
Let $\mathcal{B}\subset G$ be a nonempty open subset and $d\in\mathbb{R}$. There exist a large $\mu_1\geq1$ such that for $\mu=\mu_1$, one can find a positive constant $C$ and a large $\lambda_1$ depending only on $G$, $\mathcal{B}$, $T$, $\mu_1$, $\sigma_{ij}^0$, and $\sigma^0$ such that for any $F_0,F_1\in L^2_\mathcal{F}(0,T;L^2(G))$, and $z_0\in L^2_{\mathcal{F}_0}(\Omega;L^2(G))$, the associated weak solution $z$ of \eqref{eqqgfr} satisfies that 
\begin{align}\label{carfor5.6new}
\begin{aligned}
\mathcal{I}(d,z)\leq C\,\mathbb{E} \iint_Q    \Big\{ \lambda^d \theta^2 \gamma^d z^2 \chi_{\mathcal{B}} +\lambda^{d-3} \theta^2\gamma^{d-3}F_0^2 +\lambda^{d-1} \theta^2\gamma^{d-1}F_1^2\Big\}\,dx\,dt,
\end{aligned}
\end{align}
for all  $\lambda\geq\lambda_1(T+T^2)$.
\end{lm}

We also introduce the following backward stochastic parabolic equation
\begin{equation}\label{eqqgbc}
\begin{cases}
dz+ \displaystyle\sum_{i,j=1}^N \frac{\partial}{\partial x_i}\left(\sigma_{ij}^0(\omega,t,x)\frac{\partial z}{\partial x_j}\right) \,dt=(F_0+\nabla\cdot F)\,dt+ Z \,dW(t) & \textnormal{in}\,\,Q,\\ 
z=0 & \textnormal{on}\,\,\Sigma,\\
z(T,\cdot)=z_T & \textnormal{in}\,\, G,
\end{cases}
\end{equation}
where $z_T \in L^2_{\mathcal{F}_T}(\Omega; L^2(G))$, $F_0 \in L^2_\mathcal{F}(0,T; L^2(G))$, $F \in L^2_\mathcal{F}(0,T; L^2(G;\mathbb{R}^N))$, and the coefficients $\sigma_{ij}^0\in L^\infty_\mathcal{F}(\Omega;C^1([0,T];W^{2,\infty}(G)))$ satisfy the assumptions (H2)-(H3). 

We derive the following Carleman estimate for equation \eqref{eqqgbc}. To this end, we apply the Carleman estimate from \cite[Theorem 3.1]{elgconvec23}, established for the case $d=3$, to the equations satisfied by the weighted variables ``$(\lambda\gamma)^{\frac{d-3}{2}}z$'' and ``$(\lambda\gamma)^{\frac{d-3}{2}}Z$''.
\begin{lm}\label{thmm3.1cab}
Let $\mathcal{B}\subset G$ be a nonempty open subset and $d\in\mathbb{R}$. Then, one can find a positive constant $C$ and a large $\mu_2,\lambda_2\geq1$ depending only on $G$, $\mathcal{B}$, $T$, $\sigma_{ij}^0$, and $\sigma^0$ such that for all  $\mu=\mu_2$,  $F_0\in L^2_\mathcal{F}(0,T;L^2(G))$, $F\in L^2_\mathcal{F}(0,T;L^2(G;\mathbb{R}^N))$ and $z_T\in L^2_{\mathcal{F}_T}(\Omega;L^2(G))$, the associated weak solution $(z,Z)$ of \eqref{eqqgbc} satisfies that
\begin{align}\label{3.22carlemgenBack}
\begin{aligned}
\mathcal{I}(d,z)\leq C\, \mathbb{E} \iint_Q \Big\{ \lambda^d  \theta^2 \gamma^d z^2 \chi_{\mathcal{B}} +\lambda^{d-3} \theta^2\gamma^{d-3}F_0^2+\lambda^{d-1}\theta^2\gamma^{d-1}|F|^2+\lambda^{d-1} \theta^2\gamma^{d-1}Z^2\Big\}\,dx\,dt,
\end{aligned}
\end{align}
for all $\lambda\geq\lambda_2(T+T^2)$.
\end{lm}

In the sequel, we fix \( \mu = \overline{\mu} = \max(\mu_1, \mu_2) \), where \( \mu_1 \) and \( \mu_2 \) are the constants given in Lemmas~\ref{thm3.3} and ~\ref{thmm3.1cab}, respectively.

\subsection{Global Carleman Estimate for the Case $\beta=0$}
In this subsection, we show the following Carleman estimate for the coupled backward-forward system \eqref{adback4.77}--\eqref{adjoforr4.8}  with $\beta=0$.

\begin{thm}[Carleman estimate with $\beta=0$]\label{carlthm32}
Let us assume that \eqref{assA1}, \eqref{assA3}, \eqref{assA4}, and \eqref{assA5}  hold. Then, there exists a constant $C>0$ and a large $\overline{\lambda}\geq1$ depending only on $G$, $G_0$, $\mathcal{O}_1$, $\mathcal{O}_2$, $\overline{\mu}$, $T$, $\sigma_{ij}^m$, $\sigma_0$, $a_0$,  $a_{ij}$, $b_{ij}$, and $B_{ij}$ such that for all $\lambda\geq\overline{\lambda}$, and $q_0,v_0\in L^2_{\mathcal{F}_0}(\Omega;L^2(G))$, the associated solution $((h,k;H,K);(q,v))$ of the systems \eqref{adback4.77}-\eqref{adjoforr4.8} satisfies that
\begin{align}\label{carlestcasca}
\begin{aligned}
    &\, \mathcal{I}(2,h) + \mathcal{I}(2,k) + \mathcal{I}(3,q) + \mathcal{I}(3,v) \\
    & \leq C \,\mathbb{E} \iint_Q \Big\{ \lambda^{178}  \theta^2 \gamma^{178} h^2 \chi_{G_0} + \lambda^{88} \theta^2 \gamma^{88} (H^2+K^2)\Big\} \, dx \, dt.
    \end{aligned}
\end{align}
\end{thm}
\begin{proof}
For clarity, we organize the proof into seven steps.\\
\textbf{Step 1. Some notations and preliminaries.}\\
According to the assumption \eqref{assA1}, we set the nonempty subsets \( \widetilde{G}_i \) (\( 1 \leq i \leq 6 \)) such that
\begin{equation}\label{assuponGitilde}
\widetilde{G}_6 \Subset \widetilde{G}_5 \Subset \widetilde{G}_4 \Subset \widetilde{G}_3 \Subset \widetilde{G}_2 \Subset \widetilde{G}_1 \Subset \widetilde{G}_0\subset G_0 \cap \mathcal{O}_1,\quad\textnormal{and}\quad \widetilde{G}_0\cap\mathcal{O}_2=\emptyset,
\end{equation}
where \( \widetilde{G}_0 \) denotes the set chosen in accordance with  \eqref{assA3}.
We also define the functions \( \zeta_i \in C^{\infty}(\mathbb{R}^N) \) satisfying the following properties
\begin{align*}
\begin{aligned}
& 0 \leq \zeta_i \leq 1, \quad \zeta_i = 1 \,\, \text{in} \,\, \widetilde{G}_{7-i}, \quad \text{Supp}(\zeta_i) \subset \widetilde{G}_{6-i}, \\ 
& \frac{\Delta \zeta_i}{\zeta_i^{1/2}} \in L^\infty(G), \quad \frac{\nabla \zeta_i}{\zeta_i^{1/2}} \in L^\infty(G; \mathbb{R}^N), \quad i=1, 2,\dots, 6.
\end{aligned}
\end{align*}
Such functions \( \zeta_i \) exist (see, e.g., \cite{Tereza2000}). Next, set \( w_l = \theta^2 (\lambda \gamma)^l \) (with \( l \in \mathbb{N} \)), and observe that for large \( \lambda \), we have the following estimates
\begin{align}\label{esttmforT}
|\partial_t w_l| \leq C \theta^2 (\lambda \gamma)^{l+2}, \qquad |\nabla(w_l \zeta_i)| \leq C \theta^2 (\lambda \gamma)^{l+1} \zeta_i^{1/2}, \qquad i = 1, 2,\dots, 6.
\end{align}
Using the Carleman estimate \eqref{carfor5.6new} for the two equations of the system \eqref{adjoforr4.8} with \( \mathcal{B} = \widetilde{G}_6 \) and \( d = 3 \), and taking a large \( \lambda \), we deduce that
\begin{align}\label{estimm5.4}
\begin{aligned}
\mathcal{I}(3,q) + \mathcal{I}(3,v) \leq C \lambda^3 \mathbb{E} \int_0^T \int_{\widetilde{G}_6} \theta^2 \gamma^3 (q^2+v^2) \, dx \, dt.
\end{aligned}
\end{align}
\textbf{Step 2. Eliminating the localized integral associated to the state $q$.}\\
Recalling the assumption \eqref{assA3}, and for simplicity,  we henceforth assume that the term \( a_{21} \) satisfies
\begin{align}\label{asspa21part}
a_{21} \geq a_0 > 0 \quad \textnormal{in} \quad (0,T) \times \widetilde{G}_0, \quad \textnormal{a.s.}
\end{align}
Then notice that
\begin{align}\label{ineeforqg5}
a_0 \lambda^3 \mathbb{E} \int_0^T \int_{\widetilde{G}_6} \theta^2 \gamma^3 q^2 \, dx \, dt \leq \mathbb{E} \iint_Q w_3 \zeta_1 a_{21} q^2 \, dx \, dt.
\end{align}
Using the system \eqref{adjoforr4.8}, we compute \( d(w_3 \zeta_1 q v) \), and then we obtain
\begin{align}\label{ineqq3.155}
    \begin{aligned}
    \mathbb{E} \iint_Q w_3 \zeta_1 a_{21} q^2 \, dx \, dt &= -\mathbb{E} \iint_Q (w_3 \zeta_1)_t q v \, dx \, dt + \sum_{i,j=1}^N \mathbb{E} \iint_Q \sigma_{ij}^2 \frac{\partial (w_3 \zeta_1 q)}{\partial x_i} \frac{\partial v}{\partial x_j} \, dx \, dt \\
    &\quad - \mathbb{E} \iint_Q w_3 \zeta_1 a_{22} q v \, dx \, dt- \mathbb{E} \iint_Q w_3 \zeta_1 q (B_{21}\cdot\nabla q +B_{22}\cdot\nabla v)\, dx \, dt \\
    &\quad+ \sum_{i,j=1}^N \mathbb{E} \iint_Q \sigma_{ij}^1 \frac{\partial (w_3 \zeta_1 v)}{\partial x_i} \frac{\partial q}{\partial x_j} \, dx \, dt - \mathbb{E} \iint_Q w_3 \zeta_1 a_{11} q v \, dx \, dt \\
    &\quad - \mathbb{E} \iint_Q w_3 \zeta_1 a_{12} v^2 \, dx \, dt - \mathbb{E} \iint_Q w_3 \zeta_1 v (B_{11}\cdot\nabla q +B_{12}\cdot\nabla v )\, dx \, dt\\
    &\quad - \mathbb{E} \iint_Q w_3 \zeta_1 (b_{11}q+b_{12}v)(b_{21}q+b_{22}v) \,dx\,dt.
    \end{aligned}
\end{align}
Fix $\varepsilon>0$. Since all the terms in the right-hand side of \eqref{ineqq3.155} can be treated in a similar manner, then by using \eqref{esttmforT}, \eqref{assA4} and \eqref{assA5} together with Young's inequality, we directly obtain for a large $\lambda$,
\begin{align}\label{inqqM1m}
\begin{aligned}
\mathbb{E} \iint_Q w_3 \zeta_1 a_{21} q^2 \, dx \, dt \leq C\varepsilon& \mathcal{I}(3,q)+C\varepsilon \mathcal{I}(3,v) \\
&+ \frac{C}{\varepsilon} \bigg[\lambda^7 \mathbb{E} \int_0^T \int_{\widetilde{G}_5} \theta^2 \gamma^7 v^2 \, dx \, dt+\lambda^5 \mathbb{E} \int_0^T \int_{\widetilde{G}_5} \theta^2 \gamma^5 |\nabla v|^2 \, dx \, dt\bigg].
\end{aligned}
\end{align}
Combining \eqref{estimm5.4}, \eqref{ineeforqg5}, \eqref{inqqM1m}, and taking a sufficiently small \(\varepsilon\), we obtain the following estimate
\begin{align}\label{estimm5.4newsec}
\begin{aligned}
    \mathcal{I}(3,q) + \mathcal{I}(3,v) \leq C \bigg[ \lambda^7 \mathbb{E} \int_0^T \int_{\widetilde{G}_5} \theta^2 \gamma^7 v^2 \, dx \, dt + \lambda^5 \mathbb{E} \int_0^T \int_{\widetilde{G}_5} \theta^2 \gamma^5 |\nabla v|^2 \, dx \, dt \bigg].
\end{aligned}
\end{align}
\textbf{Step 3. Eliminating the localized integral associated with $\nabla v$.}\\
From \eqref{assmponalpha}, we have that
\begin{align}\label{ineqq3.24forgrdv}
\sigma_0\lambda^5\mathbb{E}\int_0^T\int_{\widetilde{G}_5} \theta^2\gamma^5 |\nabla v|^2\,dx\,dt\leq \mathbb{E}\iint_Q w_5\zeta_2 \sum_{i,j=1}^N \sigma_{ij}^2\frac{\partial v}{\partial x_i}\frac{\partial v}{\partial x_j}\,dx\,dt.
\end{align}
By the second equation of the system \eqref{adjoforr4.8}, and by computing \( d(w_5 \zeta_2 v^2) \), we find that
\begin{align}
    \begin{aligned}
    &2\mathbb{E}\iint_Q w_5 \zeta_2 \sum_{i,j=1}^N \sigma_{ij}^2 \frac{\partial v}{\partial x_i} \frac{\partial v}{\partial x_j} \, dx \, dt\\ 
    &= \mathbb{E} \iint_Q (w_5 \zeta_2)_t v^2 \, dx \, dt - 2 \sum_{i,j=1}^N \mathbb{E} \iint_Q \sigma_{ij}^2 v \frac{\partial (w_5 \zeta_2)}{\partial x_i} \frac{\partial v}{\partial x_j} \, dx \, dt \\
    &\quad + 2 \mathbb{E} \iint_Q w_5 \zeta_2 a_{21} q v \, dx \, dt + 2 \mathbb{E} \iint_Q w_5 \zeta_2 a_{22} v^2 \, dx \, dt \\
   &\quad + 2 \mathbb{E} \iint_Q w_5 \zeta_2 v(B_{21}\cdot\nabla q+B_{22}\cdot\nabla v) \, dx \, dt\\
   &\quad+ \mathbb{E} \iint_Q w_5 \zeta_2 (b_{21}q+b_{22} v)^2 \, dx \, dt.
    \end{aligned}
\end{align}
Let $\varepsilon>0$. Proceeding as before and using \eqref{esttmforT}, \eqref{assA4}, \eqref{assA5} and Young's inequality, and taking a large $\lambda$, we derive that
\begin{align}\label{inneqqM1r2}
2\mathbb{E}\iint_Q w_5 \zeta_2 \sum_{i,j=1}^N \sigma_{ij}^2 \frac{\partial v}{\partial x_i} \frac{\partial v}{\partial x_j} \, dx \, dt \leq C\varepsilon \mathcal{I}(3,v)+C\varepsilon \mathcal{I}(3,q) + \frac{C}{\varepsilon} \lambda^{11} \mathbb{E} \int_0^T \int_{\widetilde{G}_4} \theta^2 \gamma^{11} v^2 \, dx \, dt.
\end{align}
Combining \eqref{estimm5.4newsec}, \eqref{ineqq3.24forgrdv}, \eqref{inneqqM1r2}, and choosing a small \(\varepsilon\), we get
\begin{align}\label{estimmsecnew3.25}
    \begin{aligned}
    \mathcal{I}(3,q) + \mathcal{I}(3,v) \leq C \lambda^{11} \mathbb{E} \int_0^T \int_{\widetilde{G}_4} \theta^2 \gamma^{11} v^2 \, dx \, dt. 
    \end{aligned}
\end{align}
\textbf{Step 4. Passing from the state \( v \) to the states \( h \) and \( k \).}\\
Here, we will use the coupling term $(\chi_{\mathcal{O}_1}v+\nabla\cdot(\chi_{\mathcal{O}_2} v))$ between the equations \eqref{adback4.77} and \eqref{adjoforr4.8}, we first note that
\begin{align}\label{estineqq3.32e}
    \lambda^{11} \mathbb{E} \int_0^T \int_{\widetilde{G}_4} \theta^2 \gamma^{11} v^2 \, dx \, dt \leq \mathbb{E} \iint_Q w_{11} \zeta_3 v^2 \, dx \, dt.
\end{align}
Using Itô's formula for \( d(w_{11} \zeta_3 k v) \), we obtain that
\begin{align}\label{firsestine3.14}
    \begin{aligned}
    \mathbb{E} \iint_Q w_{11} \zeta_3 v^2 \, dx \, dt 
    &= \mathbb{E} \iint_Q (w_{11} \zeta_3)_t k v \, dx \, dt + \mathbb{E} \iint_Q w_{11} \zeta_3 a_{21} q k \, dx \, dt \\
    &\quad + \mathbb{E} \iint_Q w_{11} \zeta_3  k(B_{21}\cdot\nabla q+B_{22}\cdot\nabla v) \, dx \, dt\\
    &\quad+\sum_{i,j=1}^N \mathbb{E} \iint_Q \sigma_{ij}^2  \frac{\partial (w_{11} \zeta_3v)}{\partial x_i} \frac{\partial k}{\partial x_j} \, dx \, dt-\sum_{i,j=1}^N \mathbb{E} \iint_Q \sigma_{ij}^2 \frac{\partial (w_{11} \zeta_3k)}{\partial x_i} \frac{\partial v}{\partial x_j} \, dx \, dt\\
    &\quad - \mathbb{E} \iint_Q w_{11} \zeta_3 a_{12} h v \, dx \, dt - \mathbb{E} \iint_Q w_{11} \zeta_3 v(b_{12}H+b_{22}K) \, dx \, dt\\
    &\quad- \mathbb{E} \iint_Q (hB_{12}+kB_{22})\cdot\nabla(w_{11}\zeta_3v)  \, dx \, dt+ \mathbb{E} \iint_Q \chi_{\mathcal{O}_2} \nabla(w_{11}\zeta_3v) \cdot\nabla v  \, dx \, dt\\
    &\quad+ \mathbb{E} \iint_Q w_{11}\zeta_3(b_{21}q+b_{22}v)K   \, dx \, dt.
    \end{aligned}
\end{align}
For any $\varepsilon>0$, arguing as above, using the second assumption of \eqref{assuponGitilde} and choosing a large $\lambda$, we obtain the following estimate
\begin{align}\label{ineqqneshkk}
\begin{aligned}
\mathbb{E} \iint_Q w_{11} \zeta_3 v^2 \, dx \, dt \leq C\varepsilon& \mathcal{I}(3,v) +C\varepsilon \mathcal{I}(3,q) + \frac{C}{\varepsilon} \bigg[\lambda^{21} \mathbb{E} \iint_Q \theta^2 \gamma^{21} \zeta_3 h^2 \, dx \, dt \\
&+\lambda^{23} \mathbb{E} \iint_Q \theta^2 \gamma^{23} \zeta_3 k^2 \, dx \, dt + \lambda^{21} \mathbb{E} \iint_Q \theta^2 \gamma^{21} \zeta_3 |\nabla k|^2 \, dx \, dt\\
&+ \lambda^{19} \mathbb{E} \iint_Q \theta^2 \gamma^{19} H^2 \, dx \, dt+ \lambda^{19} \mathbb{E} \iint_Q \theta^2 \gamma^{19} K^2 \, dx \, dt\bigg].
\end{aligned}
\end{align}
Combining \eqref{estimmsecnew3.25}, \eqref{estineqq3.32e}, \eqref{ineqqneshkk}, and taking a small \(\varepsilon\), we conclude that
\begin{align}\label{esecesttsecnew3.25}
\begin{aligned}
    \mathcal{I}(3,q) + \mathcal{I}(3,v) \leq C \bigg[ &\lambda^{21} \mathbb{E} \int_0^T \int_{\widetilde{G}_3} \theta^2 \gamma^{21} h^2 \, dx \, dt + \lambda^{23} \mathbb{E} \int_0^T \int_{\widetilde{G}_3} \theta^2 \gamma^{23} k^2 \, dx \, dt\\
    &+\lambda^{21} \mathbb{E} \int_0^T \int_{\widetilde{G}_3} \theta^2 \gamma^{21} |\nabla k|^2 \, dx \, dt+ \lambda^{19} \mathbb{E} \iint_Q \theta^2 \gamma^{19} H^2 \, dx \, dt\\
    &+ \lambda^{19} \mathbb{E} \iint_Q \theta^2 \gamma^{19} K^2 \, dx \, dt\bigg].
\end{aligned}
\end{align}
We now apply the Carleman estimate \eqref{3.22carlemgenBack} to the system \eqref{adback4.77}, with \(\mathcal{B} = \widetilde{G}_3\) and \(d = 2\). This leads to the existence of a constant \(C > 0\) and a sufficiently large \(\lambda\) such that
\begin{align}\label{estt4.3f}
    \begin{aligned}
\mathcal{I}(2,h) + \mathcal{I}(2,k) \leq C \bigg[ &\lambda^2 \mathbb{E} \int_0^T \int_{\widetilde{G}_3} \theta^2 \gamma^2 h^2 \, dx \, dt + \lambda^2 \mathbb{E} \int_0^T \int_{\widetilde{G}_3} \theta^2 \gamma^2 k^2 \, dx \, dt \\
& + \lambda^{-1}\mathbb{E} \iint_Q \theta^2\gamma^{-1} v^2 \, dx \, dt+ \lambda\mathbb{E} \iint_Q \theta^2\gamma |\nabla v|^2 \, dx \, dt \\
&+ \lambda \mathbb{E} \iint_Q \theta^2 \gamma H^2 \, dx \, dt+ \lambda \mathbb{E} \iint_Q \theta^2 \gamma K^2 \, dx \, dt \bigg].
    \end{aligned}
\end{align}
Combining \eqref{estt4.3f} and \eqref{esecesttsecnew3.25}, and taking a large \(\lambda\), we get that
\begin{align}\label{estt4.3fnewone}
    \begin{aligned}
    &\,\mathcal{I}(2,h) + \mathcal{I}(2,k) + \mathcal{I}(3,q) + \mathcal{I}(3,v) \\
    &\leq C \bigg[\lambda^{21} \mathbb{E} \int_0^T \int_{\widetilde{G}_3} \theta^2 \gamma^{21} h^2 \, dx \, dt + \lambda^{23} \mathbb{E} \int_0^T \int_{\widetilde{G}_3} \theta^2 \gamma^{23} k^2 \, dx \, dt \\
    & \hspace{1cm}+ \lambda^{19} \mathbb{E} \iint_Q \theta^2 \gamma^{19} H^2 \, dx \, dt + \lambda^{19} \mathbb{E} \iint_Q \theta^2 \gamma^{19} K^2 \, dx \, dt\\
    &\hspace{1cm}+\lambda^{21} \mathbb{E} \int_0^T \int_{\widetilde{G}_3} \theta^2 \gamma^{21} |\nabla k|^2 \, dx \, dt\bigg].
    \end{aligned}
\end{align}
\textbf{Step 5. Eliminating the localized integral associated with \(\nabla k\).}\\
Notice that
\begin{align}\label{esiforgradh2}
    \sigma_0 \lambda^{21} \mathbb{E} \int_0^T \int_{\widetilde{G}_3} \theta^2 \gamma^{21} |\nabla k|^2 \, dx \, dt \leq \mathbb{E} \iint_Q w_{21} \zeta_4 \sum_{i,j=1}^N \sigma_{ij}^2 \frac{\partial k}{\partial x_i} \frac{\partial k}{\partial x_j} \, dx \, dt.
\end{align}
By the first equation of \eqref{adback4.77}, and computing \( d(w_{21} \zeta_4 k^2) \), we obtain
\begin{align}\label{mainestinfirst2}
    \begin{aligned}
    &2 \mathbb{E} \iint_Q w_{21} \zeta_4 \sum_{i,j=1}^N \sigma_{ij}^2 \frac{\partial k}{\partial x_i} \frac{\partial k}{\partial x_j} \, dx \, dt\\
    &= - \mathbb{E} \iint_Q (w_{21} \zeta_4)_t k^2 \, dx \, dt - 2 \sum_{i,j=1}^N \mathbb{E} \iint_Q \sigma_{ij}^2 k \frac{\partial (w_{21} \zeta_4)}{\partial x_i} \frac{\partial k}{\partial x_j} \, dx \, dt \\
    & \quad  + 2 \mathbb{E} \iint_Q w_{21} \zeta_4 a_{12} h k \, dx \, dt + 2 \mathbb{E} \iint_Q w_{21} \zeta_4 a_{22} k^2 \, dx \, dt\\
    & \quad  + 2 \mathbb{E} \iint_Q w_{21} \zeta_4 b_{12} kH \, dx \, dt + 2 \mathbb{E} \iint_Q w_{21} \zeta_4 b_{22} k K \, dx \, dt\\
    & \quad + 2\mathbb{E} \iint_Q (h B_{12}+k B_{22})\cdot\nabla(w_{21} \zeta_4 k) \, dx \, dt \\
    & \quad+2 \mathbb{E} \iint_Q w_{21} \zeta_4 k\chi_{\mathcal{O}_1}v \,dx\,dt -2 \mathbb{E} \iint_Q \chi_{\mathcal{O}_2}\nabla(w_{21} \zeta_4 k)\cdot\nabla v \, dx \, dt\\
    &\quad- \mathbb{E} \iint_Q w_{21} \zeta_4 K^2 \, dx \, dt.
    \end{aligned}
\end{align}
For any $\varepsilon>0$, using again the second assumption of \eqref{assuponGitilde}, we have for a large enough $\lambda$,
\begin{align}\label{ineqqnes336}
\begin{aligned}
    2 \mathbb{E} \iint_Q w_{21} \zeta_4 \sum_{i,j=1}^N \sigma_{ij}^2 \frac{\partial k}{\partial x_i} \frac{\partial k}{\partial x_j} \, dx \, dt\leq C\varepsilon& \mathcal{I}(2,k)+C\varepsilon \mathcal{I}(3,v)+\frac{C}{\varepsilon}\bigg[\lambda^{42}\mathbb{E}\int_0^T \int_{\widetilde{G}_2} \theta^2\gamma^{42} h^2 \,dx\,dt\\
    &+\lambda^{44}\mathbb{E}\int_0^T \int_{\widetilde{G}_2} \theta^2\gamma^{44} k^2 \,dx\,dt+\lambda^{40}\mathbb{E}\iint_Q \theta^2\gamma^{40} K^2 \,dx\,dt\\
    &+\lambda^{40}\mathbb{E}\iint_Q \theta^2\gamma^{40} H^2 \,dx\,dt\bigg].
    \end{aligned}
\end{align}
Combining \eqref{estt4.3fnewone}, \eqref{esiforgradh2}, \eqref{mainestinfirst2}, \eqref{ineqqnes336}, choosing a small $\varepsilon$, we get that
\begin{align}\label{estt4.3fnewone2}
    \begin{aligned}
    &\,\mathcal{I}(2,h) + \mathcal{I}(2,k) + \mathcal{I}(3,q) + \mathcal{I}(3,v) \\
    &\leq C \bigg[\lambda^{42} \mathbb{E} \int_0^T \int_{\widetilde{G}_2} \theta^2 \gamma^{42} h^2 \, dx \, dt + \lambda^{44} \mathbb{E} \int_0^T \int_{\widetilde{G}_2} \theta^2 \gamma^{44} k^2 \, dx \, dt \\
    & \hspace{1cm}+ \lambda^{40} \mathbb{E} \iint_Q \theta^2 \gamma^{40} H^2 \, dx \, dt + \lambda^{40} \mathbb{E} \iint_Q \theta^2 \gamma^{40} K^2 \, dx \, dt\bigg].
    \end{aligned}
\end{align}
\textbf{Step 6. Eliminating the localized integral associated to the state $k$.}\\
From \eqref{asspa21part}, we observe that
\begin{align}\label{firstestfork}
    a_0 \lambda^{44} \mathbb{E} \int_0^T \int_{\widetilde{G}_2} \theta^2 \gamma^{44} k^2 \, dx \, dt \leq \mathbb{E} \iint_Q w_{44} \zeta_5 a_{21} k^2 \, dx \, dt.
\end{align}
Using the system \eqref{adback4.77} and applying Itô's formula for \( d(w_{44} \zeta_5 h k) \), we arrive at
\begin{align}\label{mainestinfirst}
    \begin{aligned}
    \mathbb{E} \iint_Q w_{44} \zeta_5 a_{21} k^2 \, dx \, dt
    &= \mathbb{E} \iint_Q (w_{44} \zeta_5)_t h k \, dx \, dt + \sum_{i,j=1}^N \mathbb{E} \iint_Q \sigma_{ij}^1 \frac{\partial (w_{44} \zeta_5 k)}{\partial x_i} \frac{\partial h}{\partial x_j} \, dx \, dt \\
    & \quad - \mathbb{E} \iint_Q w_{44} \zeta_5 a_{11} h k \, dx \, dt- \mathbb{E} \iint_Q (h B_{11}+kB_{21})\cdot\nabla(w_{44} \zeta_5 k) \, dx \, dt \\
    & \quad + \sum_{i,j=1}^N \mathbb{E} \iint_Q \sigma_{ij}^2 \frac{\partial (w_{44} \zeta_5 h)}{\partial x_i} \frac{\partial k}{\partial x_j} \, dx \, dt - \mathbb{E} \iint_Q w_{44} \zeta_5 a_{12} h^2 \, dx \, dt\\
    &\quad - \mathbb{E} \iint_Q w_{44} \zeta_5 h(b_{12}H+b_{22}K)\, dx \, dt- \mathbb{E} \iint_Q w_{44} \zeta_5 k (b_{11}H+b_{21}K)\, dx \, dt\\
    & \quad - \mathbb{E} \iint_Q w_{44} \zeta_5 a_{22} h k \, dx \, dt  - \mathbb{E} \iint_Q (h B_{12}+kB_{22})\cdot\nabla(w_{44} \zeta_5 h) \, dx \, dt\\
    & \quad - \mathbb{E} \iint_Q w_{44} \zeta_5 \chi_{\mathcal{O}_1} v h \, dx \, dt+ \mathbb{E} \iint_Q \chi_{\mathcal{O}_2} \nabla(w_{44} \zeta_5 h)\cdot \nabla v \, dx \, dt\\
    &\quad + \mathbb{E} \iint_Q w_{44} \zeta_5 H K \, dx \, dt.\\
    \end{aligned}
\end{align}
Fix $\varepsilon > 0$. By \eqref{esttmforT} along with Young's inequality, \eqref{assA4}, \eqref{assA5} and the second assumption of \eqref{assuponGitilde}, then by taking a large enough $\lambda$, we derive the following estimate
\begin{align}\label{ineqq3.2f}
\begin{aligned}
\mathbb{E} \iint_Q w_{44} \zeta_4 a_{21} k^2 \, dx \, dt \leq C\varepsilon& \mathcal{I}(2,k)+C\varepsilon \mathcal{I}(2,h)+C\varepsilon \mathcal{I}(3,v) \\
 &+ \frac{C}{\varepsilon}\bigg[\lambda^{90} \mathbb{E} \int_0^T \int_{\widetilde{G}_1} \theta^2 \gamma^{90} h^2 \, dx \, dt+\lambda^{88} \mathbb{E} \int_0^T \int_{\widetilde{G}_1} \theta^2 \gamma^{88} |\nabla h|^2 \, dx \, dt\\
 &\qquad\qquad+\lambda^{86} \mathbb{E} \iint_Q \theta^2 \gamma^{86} H^2 \, dx \, dt + \lambda^{44} \mathbb{E} \iint_Q \theta^2 \gamma^{44} K^2 \, dx \, dt\bigg].
 \end{aligned}
\end{align}
Combining \eqref{estt4.3fnewone}, \eqref{firstestfork}, \eqref{mainestinfirst}, \eqref{ineqq3.2f}, and selecting a small \(\varepsilon\), we obtain that
\begin{align}\label{estt4.3fnewonelast}
    \begin{aligned}
    &\, \mathcal{I}(2,h) + \mathcal{I}(2,k) + \mathcal{I}(3,q) + \mathcal{I}(3,v) \\
    & \leq C \bigg[ \lambda^{90} \mathbb{E} \int_0^T \int_{\widetilde{G}_1} \theta^2 \gamma^{90} h^2 \, dx \, dt + \lambda^{86} \mathbb{E} \iint_Q \theta^2 \gamma^{86} H^2 \, dx \, dt \\
    & \hspace{1cm} + \lambda^{44} \mathbb{E} \iint_Q \theta^2 \gamma^{44} K^2 \, dx \, dt + \lambda^{88} \mathbb{E} \int_0^T \int_{\widetilde{G}_1} \theta^2 \gamma^{88} |\nabla h|^2 \, dx \, dt\bigg].
    \end{aligned}
\end{align}
\textbf{Step 7. Eliminating the localized integral associated with \(\nabla h\).}\\
Notice that
\begin{align}\label{esiforgradh}
    \sigma_0 \lambda^{88} \mathbb{E} \int_0^T \int_{\widetilde{G}_1} \theta^2 \gamma^{88} |\nabla h|^2 \, dx \, dt \leq \mathbb{E} \iint_Q w_{88} \zeta_6 \sum_{i,j=1}^N \sigma_{ij}^1 \frac{\partial h}{\partial x_i} \frac{\partial h}{\partial x_j} \, dx \, dt.
\end{align}
By the first equation of \eqref{adback4.77}, and computing \( d(w_{88} \zeta_6 h^2) \), we obtain
\begin{align}\label{mainestinfirst2sec}
    \begin{aligned}
    &2 \mathbb{E} \iint_Q w_{88} \zeta_6 \sum_{i,j=1}^N \sigma_{ij}^1 \frac{\partial h}{\partial x_i} \frac{\partial h}{\partial x_j} \, dx \, dt\\
    &= - \mathbb{E} \iint_Q (w_{88} \zeta_6)_t h^2 \, dx \, dt - 2 \sum_{i,j=1}^N \mathbb{E} \iint_Q \sigma_{ij}^1 h \frac{\partial (w_{88} \zeta_6)}{\partial x_i} \frac{\partial h}{\partial x_j} \, dx \, dt \\
    & \quad + 2 \mathbb{E} \iint_Q w_{88} \zeta_6 a_{11} h^2 \, dx \, dt + 2 \mathbb{E} \iint_Q w_{88} \zeta_6 a_{21} h k \, dx \, dt \\
    & \quad + 2 \mathbb{E} \iint_Q w_{88} \zeta_6 h(b_{11}H+b_{21}K) \, dx \, dt + 2\mathbb{E} \iint_Q (h B_{11}+kB_{21})\cdot\nabla(w_{88} \zeta_6 h) \, dx \, dt \\
    & \quad- \mathbb{E} \iint_Q w_{88} \zeta_6 H^2 \, dx \, dt.
    \end{aligned}
\end{align}
For any $\varepsilon>0$, recalling \eqref{esttmforT}, using Young's inequality and \eqref{assA5}, and choosing a large $\lambda$, we obtain that
\begin{align}\label{ineqqse343l}
\begin{aligned}
&2 \mathbb{E} \iint_Q w_{88} \zeta_6 \sum_{i,j=1}^N \sigma_{ij}^1 \frac{\partial h}{\partial x_i} \frac{\partial h}{\partial x_j} \, dx \, dt\\
&\leq C\varepsilon \mathcal{I}(2,h)+C\varepsilon \mathcal{I}(2,k)+ \frac{C}{\varepsilon} \bigg[\lambda^{178} \mathbb{E} \int_0^T \int_{\widetilde{G}_0} \theta^2 \gamma^{178} h^2 \, dx \, dt+\lambda^{88} \mathbb{E} \iint_Q \theta^2 \gamma^{88} (H^2+K^2) \, dx \, dt\bigg].
\end{aligned}
\end{align}
Combining \eqref{estt4.3fnewonelast}, \eqref{esiforgradh}, \eqref{mainestinfirst2sec}, \eqref{ineqqse343l}, and choosing a small \(\varepsilon\), we conclude that
\begin{align*}
    \begin{aligned}
    &\, \mathcal{I}(2,h) + \mathcal{I}(2,k) + \mathcal{I}(3,q) + \mathcal{I}(3,v) \\
    & \leq C \bigg[ \lambda^{178} \mathbb{E} \int_0^T \int_{\widetilde{G}_0} \theta^2 \gamma^{178} h^2 \, dx \, dt + \lambda^{88} \mathbb{E} \iint_Q \theta^2 \gamma^{88} (H^2+K^2) \, dx \, dt \bigg],
    \end{aligned}
\end{align*}
which gives the desired Carleman estimate \eqref{carlestcasca}. This completes the proof of Theorem \ref{carlthm32}.
\end{proof}
\subsection{Global Carleman Estimate for the Case $\beta\in(0,1)$}
In this subsection, we prove the following Carleman estimate for the coupled system \eqref{adback4.77}--\eqref{adjoforr4.8}.

\begin{thm}[Carleman estimate with $\beta\in(0,1)$]\label{carlthm3233.4}
Let us assume that \eqref{assA2} holds. Then, there exists a constant $C>0$ and a large $\overline{\lambda}\geq1$ depending only on $G$, $G_i$, $\mathcal{O}_i$, $\overline{\mu}$, $T$, $\sigma_{ij}^m$, $\sigma_0$, $\beta$,  $a_{ij}$, $b_{ij}$ and $B_{ij}$ such that for all $\lambda\geq\overline{\lambda}$, and $q_0,v_0\in L^2_{\mathcal{F}_0}(\Omega;L^2(G))$, the associated solution $((h,k;H,K);(q,v))$ of the system \eqref{adback4.77}-\eqref{adjoforr4.8} satisfies that 
\begin{align}\label{carlestcasca33.4}
\begin{aligned}
    &\,\mathcal{I}(2,h) + \mathcal{I}(2,k) + \mathcal{I}(3,q) + \mathcal{I}(3,v) \\
    &\leq C \, \mathbb{E} \iint_Q \Big\{\lambda^{12}\theta^2 \gamma^{12} (h^2\chi_{G_0}+k^2\chi_{G_1}) + \lambda^{5}  \theta^2 \gamma^{5} (H^2+K^2) \Big\}\, dx \, dt.
    \end{aligned}
\end{align}
\end{thm}

\begin{proof}
We divide the proof into three steps.\\
\textbf{Step 1. Some notations and preliminaries.}\\
According to the condition \eqref{assA2}, we consider the nonempty subsets \( \widetilde{G}_i \) (\( i=0,1,2 \)) such that
\begin{equation}\label{assonGitilde}
\widetilde{G}_2 \Subset \widetilde{G}_1 \Subset \widetilde{G}_0 \subset G_0 \cap G_1 \cap \mathcal{O}_1\cap \mathcal{O}_3,\quad\textnormal{and}\quad \widetilde{G}_0\cap(\mathcal{O}_2\cup\mathcal{O}_4)=\emptyset.
\end{equation}
We also define the functions \( \zeta_i \in C^{\infty}(\mathbb{R}^N) \) satisfying the following properties
\begin{align*}
\begin{aligned}
& 0 \leq \zeta_i \leq 1, \quad \zeta_i = 1 \,\, \text{in} \,\, \widetilde{G}_{3-i}, \quad \text{Supp}(\zeta_i) \subset \widetilde{G}_{2-i}, \\ 
& \frac{\Delta \zeta_i}{\zeta_i^{1/2}} \in L^\infty(G), \quad \frac{\nabla \zeta_i}{\zeta_i^{1/2}} \in L^\infty(G; \mathbb{R}^N), \quad i=1, 2.
\end{aligned}
\end{align*}
Next, set \( w_l = \theta^2 (\lambda \gamma)^l \) (with \( l \in \mathbb{N} \)), and observe that for large \( \lambda \), we have the following estimates
\begin{align}\label{esttmforTse}
|\partial_t w_l| \leq C \theta^2 (\lambda \gamma)^{l+2}, \qquad |\nabla(w_l \zeta_i)| \leq C \theta^2 (\lambda \gamma)^{l+1} \zeta_i^{1/2}, \qquad i = 1, 2.
\end{align}
Using the Carleman estimate \eqref{carfor5.6new} for  the system \eqref{adjoforr4.8} with \( \mathcal{B} = \widetilde{G}_2 \) and \( d = 3 \), and taking a large \( \lambda \), we deduce that
\begin{align}\label{estimm5.4ses}
\begin{aligned}
\mathcal{I}(3,q) + \mathcal{I}(3,v) \leq C \,\lambda^3 \mathbb{E} \int_0^T \int_{\widetilde{G}_2} \theta^2 \gamma^3 (q^2+v^2) \, dx \, dt.
\end{aligned}
\end{align}
\textbf{Step 2. Passing from the states \( q \) and $v$ to the states \( h \) and \( k \).}\\
Using the coupling terms  $\beta(\chi_{\mathcal{O}_3}q+\nabla\cdot(\chi_{\mathcal{O}_4}\nabla q))$ and $(1-\beta)(\chi_{\mathcal{O}_1}v+\nabla\cdot(\chi_{\mathcal{O}_2}\nabla v))$, we first note that 
\begin{align}\label{2estineqq3.32eses}
    \lambda^{3} \mathbb{E} \int_0^T \int_{\widetilde{G}_2} \theta^2 \gamma^{3} (q^2+v^2) \, dx \, dt \leq \mathbb{E} \iint_Q w_{3} \zeta_1 (q^2+v^2) \, dx \, dt.
\end{align}
Using \eqref{assonGitilde} and Itô's formula for \( d[w_{3} \zeta_1 (h q+kv)] \), we obtain that 
\begin{align}\label{2firsestine3.14ses}
    \begin{aligned}
    &\min(\beta,1-\beta)\,\mathbb{E} \iint_Q w_{3} \zeta_1 (q^2+v^2) \, dx \, dt \\
    &\leq\beta\mathbb{E} \iint_Q w_{3} \zeta_1 q^2 \, dx \, dt + (1-\beta)\mathbb{E} \iint_Q w_{3} \zeta_1 v^2 \, dx \, dt \\
    &= \mathbb{E} \iint_Q (w_{3} \zeta_1)_t hq \, dx \, dt +\mathbb{E} \iint_Q (w_{3} \zeta_1)_t k v \, dx \, dt\\
    &\quad+\sum_{i,j=1}^N \mathbb{E} \iint_Q \sigma_{ij}^1 q\frac{\partial (w_{3} \zeta_1)}{\partial x_i} \frac{\partial h}{\partial x_j} \, dx \, dt-\sum_{i,j=1}^N \mathbb{E} \iint_Q \sigma_{ij}^1 h\frac{\partial (w_{3} \zeta_1)}{\partial x_i} \frac{\partial q}{\partial x_j} \, dx \, dt\\
    &\quad - \mathbb{E} \iint_Q  kq B_{21}\cdot\nabla(w_{3} \zeta_1) \, dx \, dt+ \beta\mathbb{E} \iint_Q  \chi_{\mathcal{O}_4}\nabla(w_3\zeta_1q)\cdot\nabla q \, dx \, dt\\
    &\quad- \mathbb{E} \iint_Q  hq B_{11}\cdot\nabla(w_{3} \zeta_1) \, dx \, dt+\mathbb{E} \iint_Q  w_{3} \zeta_1 h B_{12}\cdot\nabla v \, dx \, dt \\
    &\quad+\sum_{i,j=1}^N \mathbb{E} \iint_Q \sigma_{ij}^2 v\frac{\partial (w_{3} \zeta_1)}{\partial x_i} \frac{\partial k}{\partial x_j} \, dx \, dt-\sum_{i,j=1}^N \mathbb{E} \iint_Q \sigma_{ij}^2 k\frac{\partial (w_{3} \zeta_1)}{\partial x_i} \frac{\partial v}{\partial x_j} \, dx \, dt\\
    &\quad - \mathbb{E} \iint_Q  h vB_{12}\cdot\nabla(w_{3} \zeta_1) \, dx \, dt- \mathbb{E} \iint_Q  kv B_{22}\cdot\nabla(w_{3} \zeta_1) \, dx \, dt\\
    &\quad+ (1-\beta)\mathbb{E} \iint_Q  \chi_{\mathcal{O}_2}\nabla(w_3\zeta_1v)\cdot\nabla v \, dx \, dt.
    \end{aligned}
\end{align}
By the second assumption of \eqref{assonGitilde} and applying Young's inequality, it is straightforward to deduce the following estimate: For any $\varepsilon>0$, and a large enough $\lambda$,
\begin{align}\label{intggII2ine}
\begin{aligned}
\mathbb{E} \iint_Q w_{3} \zeta_1 (q^2+v^2) \, dx \, dt \leq C\varepsilon& \mathcal{I}(3,q)+C\varepsilon \mathcal{I}(3,v)\\
&+\frac{C}{\varepsilon}\bigg[\lambda^7\mathbb{E}\int_0^T \int_{\widetilde{G}_1} \theta^2\gamma^7 h^2 \,dx\,dt+\lambda^7\mathbb{E}\int_0^T \int_{\widetilde{G}_1} \theta^2\gamma^7 k^2 \,dx\,dt\\
&\qquad\;+\lambda^5\mathbb{E}\int_0^T \int_{\widetilde{G}_1} \theta^2\gamma^5 |\nabla h|^2 \,dx\,dt+\lambda^5\mathbb{E}\int_0^T \int_{\widetilde{G}_1} \theta^2\gamma^5 |\nabla k|^2 \,dx\,dt\bigg].
\end{aligned}
\end{align}
Combining \eqref{estimm5.4ses}, \eqref{2estineqq3.32eses}, \eqref{2firsestine3.14ses}, \eqref{intggII2ine}, and taking a small \(\varepsilon\), we conclude that
\begin{align}\label{esecesttsecnew3.25ses}
\begin{aligned}
    \mathcal{I}(3,q) + \mathcal{I}(3,v) \leq C \bigg[ &\lambda^{7} \mathbb{E} \int_0^T \int_{\widetilde{G}_1} \theta^2 \gamma^{7} (h^2+k^2) \, dx \, dt\\
    &+\lambda^{5} \mathbb{E} \int_0^T \int_{\widetilde{G}_1} \theta^2 \gamma^{5} (|\nabla h|^2 +|\nabla k|^2) \, dx \, dt\bigg].
\end{aligned}
\end{align}
We now apply the Carleman estimate \eqref{3.22carlemgenBack} to the system \eqref{adback4.77}, with \(\mathcal{B} = \widetilde{G}_1\) and \(d = 2\). This leads to the existence of a constant \(C > 0\) and a sufficiently large \(\lambda\) such that
\begin{align}\label{estt4.3fses}
    \begin{aligned}
\mathcal{I}(2,h) + \mathcal{I}(2,k) \leq C \bigg[ &\lambda^2 \mathbb{E} \int_0^T \int_{\widetilde{G}_1} \theta^2 \gamma^2 (h^2+k^2) \, dx \, dt   + \lambda^{-1}\mathbb{E} \iint_Q \theta^2\gamma^{-1} (q^2+v^2) \, dx \, dt \\
&+ \lambda \mathbb{E} \iint_Q \theta^2\gamma (|\nabla q|^2+|\nabla v|^2) \, dx \, dt+ \lambda \mathbb{E} \iint_Q \theta^2 \gamma (H^2+K^2) \, dx \, dt\bigg].
    \end{aligned}
\end{align}
Combining \eqref{estt4.3fses} and \eqref{esecesttsecnew3.25ses}, and taking a large \(\lambda\), we get that
\begin{align}\label{estt4.3fnewoneses}
    \begin{aligned}
    &\,\mathcal{I}(2,h) + \mathcal{I}(2,k) + \mathcal{I}(3,q) + \mathcal{I}(3,v) \\
    &\leq C \bigg[ \lambda^{7} \mathbb{E} \int_0^T \int_{\widetilde{G}_1} \theta^2 \gamma^{7} (h^2+k^2) \, dx \, dt+\lambda^{5} \mathbb{E} \int_0^T \int_{\widetilde{G}_1} \theta^2 \gamma^{5} (|\nabla h|^2+|\nabla k|^2) \, dx \, dt\\
    &\qquad\quad+ \lambda \mathbb{E} \iint_Q \theta^2 \gamma (H^2+K^2) \, dx \, dt \bigg].
    \end{aligned}
\end{align}
\textbf{Step 3. Eliminating the localized integral associated with \(\nabla h\) and \(\nabla k\).}\\
Notice that 
\begin{align}\label{esiforgradhses}
\begin{aligned}
    \sigma_0 \lambda^{5} \mathbb{E} \int_0^T \int_{\widetilde{G}_1} \theta^2 \gamma^{5} (|\nabla h|^2+|\nabla k|^2) \, dx \, dt &\leq \mathbb{E} \iint_Q w_{5} \zeta_2 \sum_{i,j=1}^N \sigma_{ij}^1 \frac{\partial h}{\partial x_i} \frac{\partial h}{\partial x_j} \, dx \, dt\\
    &\quad+\mathbb{E} \iint_Q w_{5} \zeta_2 \sum_{i,j=1}^N \sigma_{ij}^2 \frac{\partial k}{\partial x_i} \frac{\partial k}{\partial x_j} \, dx \, dt.
    \end{aligned}
\end{align}
From the system \eqref{adback4.77}, computing \( d[w_{5} \zeta_2 (h^2+k^2)] \), we obtain
\begin{align}\label{mainestinfirst2ses}
    \begin{aligned}
    &2 \mathbb{E} \iint_Q w_{5} \zeta_2 \sum_{i,j=1}^N \sigma_{ij}^1 \frac{\partial h}{\partial x_i} \frac{\partial h}{\partial x_j} \, dx \, dt+2 \mathbb{E} \iint_Q w_{5} \zeta_2 \sum_{i,j=1}^N \sigma_{ij}^2 \frac{\partial k}{\partial x_i} \frac{\partial k}{\partial x_j} \, dx \, dt\\
    &= - \mathbb{E} \iint_Q (w_{5} \zeta_2)_t h^2 \, dx \, dt - 2 \sum_{i,j=1}^N \mathbb{E} \iint_Q \sigma_{ij}^1 h \frac{\partial (w_{5} \zeta_2)}{\partial x_i} \frac{\partial h}{\partial x_j} \, dx \, dt \\
    & \quad + 2 \mathbb{E} \iint_Q w_{5} \zeta_2 a_{11} h^2 \, dx \, dt + 2 \mathbb{E} \iint_Q w_{5} \zeta_2 a_{21} h k \, dx \, dt \\
    & \quad + 2\mathbb{E} \iint_Q (h B_{11}+k B_{21})\cdot\nabla(w_{5} \zeta_2 h) \, dx \, dt + 2 \mathbb{E} \iint_Q w_{5} \zeta_2 h(b_{11} H+b_{21}K) \, dx \, dt\\
    & \quad+ 2\beta\mathbb{E} \iint_Q  \Big(\chi_{\mathcal{O}_3} w_{5} \zeta_2 hq-\chi_{\mathcal{O}_4} \nabla(w_{5} \zeta_2 h)\cdot\nabla q\Big) \, dx \, dt- \mathbb{E} \iint_Q w_{5} \zeta_2 H^2 \, dx \, dt \\
    &\quad- \mathbb{E} \iint_Q (w_{5} \zeta_2)_t k^2 \, dx \, dt - 2 \sum_{i,j=1}^N \mathbb{E} \iint_Q \sigma_{ij}^2 k \frac{\partial (w_{5} \zeta_2)}{\partial x_i} \frac{\partial k}{\partial x_j} \, dx \, dt \\
    & \quad + 2 \mathbb{E} \iint_Q w_{5} \zeta_2 a_{22} k^2 \, dx \, dt + 2 \mathbb{E} \iint_Q w_{5} \zeta_2 a_{12} h k \, dx \, dt \\
    & \quad + 2\mathbb{E} \iint_Q (k B_{22}+h B_{12})\cdot\nabla(w_{5} \zeta_2 k) \, dx \, dt+ 2 \mathbb{E} \iint_Q w_{5} \zeta_2 k(b_{12} H+b_{22}K) \, dx \, dt\\
    & \quad+2(1-\beta)\mathbb{E} \iint_Q  \Big(\chi_{\mathcal{O}_1} w_{5} \zeta_2 kv-\chi_{\mathcal{O}_2} \nabla(w_{5} \zeta_2 k)\cdot\nabla v\Big) \, dx \, dt- \mathbb{E} \iint_Q w_{5} \zeta_2 K^2 \, dx \, dt.
    \end{aligned}
\end{align}
Recalling \eqref{esttmforTse}, the second assumption of \eqref{assonGitilde} and by Young's inequality, we find that for any \( \varepsilon > 0 \), 
\begin{align}\label{ineqI1estses17in}
\begin{aligned}
&2 \mathbb{E} \iint_Q w_{5} \zeta_2 \sum_{i,j=1}^N \sigma_{ij}^1 \frac{\partial h}{\partial x_i} \frac{\partial h}{\partial x_j} \, dx \, dt+2 \mathbb{E} \iint_Q w_{5} \zeta_2 \sum_{i,j=1}^N \sigma_{ij}^2 \frac{\partial k}{\partial x_i} \frac{\partial k}{\partial x_j} \, dx \, dt\\
&\leq C\varepsilon \mathcal{I}(2,h)+C\varepsilon \mathcal{I}(2,k)+C\varepsilon \mathcal{I}(3,q)+C\varepsilon \mathcal{I}(3,v)\\
&\quad+ \frac{C}{\varepsilon}\bigg[ \lambda^{12} \mathbb{E} \int_0^T \int_{\widetilde{G}_0} \theta^2 \gamma^{12} (h^2+k^2) \, dx \, dt+\lambda^{5} \mathbb{E} \iint_Q \theta^2 \gamma^{5} (H^2+K^2) \, dx \, dt\bigg].
\end{aligned}
\end{align}
Combining \eqref{estt4.3fnewoneses}, \eqref{esiforgradhses}, \eqref{mainestinfirst2ses}, \eqref{ineqI1estses17in}, and choosing a small \(\varepsilon\), we conclude that
\begin{align*}
    \begin{aligned}
    &\, \mathcal{I}(2,h) + \mathcal{I}(2,k) + \mathcal{I}(3,q) + \mathcal{I}(3,v) \\
    & \leq C \bigg[ \lambda^{12} \mathbb{E} \int_0^T \int_{\widetilde{G}_0} \theta^2 \gamma^{12} (h^2+k^2) \, dx \, dt+\lambda^{5} \mathbb{E} \iint_Q \theta^2 \gamma^{5} (H^2+K^2) \, dx \, dt\bigg],
    \end{aligned}
\end{align*}
which gives the desired Carleman estimate \eqref{carlestcasca33.4}. This completes the proof of Theorem \ref{carlthm3233.4}.
\end{proof}

\section{Observability Inequality}\label{sec4}

Now, by applying the global Carleman estimate \eqref{carlestcasca}, we prove the following observability inequality for the coupled system \eqref{adback4.77}-\eqref{adjoforr4.8} with $\beta=0$.

\begin{prop}[Case $\beta=0$]\label{propo5.1obseine}
We assume that  \eqref{assA1}, \eqref{assA3}, \eqref{assA4}, and \eqref{assA5} hold. Then, there exist positive constants \( C \) and \( M \) depending only on \( G \), \( G_0 \), \( \mathcal{O}_1 \), \( \mathcal{O}_2 \), \( T \), \( \sigma_{ij}^m \), \( \sigma_0 \), \( a_0 \), \( a_{ij} \), $b_{ij}$, and \( B_{ij} \) such that for any \( q_0, v_0 \in L^2_{\mathcal{F}_0}(\Omega;L^2(G)) \), the solution \( ((h,k;H,K);(q,v)) \) of the coupled systems \eqref{adback4.77}-\eqref{adjoforr4.8} satisfies that
\begin{align}\label{obseine4.9}
\begin{aligned}
\mathbb{E} \iint_Q \exp\left(-Mt^{-1}\right) \left( h^2 + k^2 \right) \, dx \, dt 
&\leq C \,\mathbb{E} \iint_Q \left(h^2\chi_{G_0} + H^2 + K^2 \right) \, dx \, dt.
\end{aligned}
\end{align}
\end{prop}
\begin{proof}
Here, we fix \( \lambda = \overline{\lambda} \) as given in Theorem \ref{carlthm32}, and begin by deriving the standard energy estimate for the forward system \eqref{adjoforr4.8}. Let \( t_1, t_2 \in (0,T) \) such that \( t_1 < t_2 \). Using Itô's formula for \( d(q^2 + v^2) \), integrating the resulting equality over the interval \( (t_1, t_2) \), and taking the expectation on both sides, we obtain
\begin{align*}
\mathbb{E}\int_{t_1}^{t_2}\int_G d(q^2 + v^2) \,dx 
&= -2\mathbb{E}\int_{t_1}^{t_2}\int_G \sum_{i,j=1}^N \sigma_{ij}^1\frac{\partial q}{\partial x_i}\frac{\partial q}{\partial x_j}\,dx\,dt + 2\mathbb{E}\int_{t_1}^{t_2}\int_G a_{11} q^2 \,dx\,dt\\
&\hspace{0.4cm} + 2\mathbb{E}\int_{t_1}^{t_2}\int_G a_{12} qv \,dx\,dt
+ 2\mathbb{E}\int_{t_1}^{t_2}\int_G q B_{11}\cdot\nabla q \,dx\,dt\\
&\hspace{0.4cm}+ 2\mathbb{E}\int_{t_1}^{t_2}\int_G q B_{12}\cdot\nabla v \,dx\,dt+\mathbb{E}\int_{t_1}^{t_2}\int_G (b_{11}q+b_{12}v)^2 \,dx\,dt
\\
&\hspace{0.4cm}
-2\mathbb{E}\int_{t_1}^{t_2}\int_G \sum_{i,j=1}^N \sigma_{ij}^2\frac{\partial v}{\partial x_i}\frac{\partial v}{\partial x_j} \,dx\,dt+ 2\mathbb{E}\int_{t_1}^{t_2}\int_G v B_{21}\cdot\nabla q \,dx\,dt\\
&\hspace{0.4cm}+ 2\mathbb{E}\int_{t_1}^{t_2}\int_G a_{21} qv \,dx\,dt + 2\mathbb{E}\int_{t_1}^{t_2}\int_G a_{22} v^2 \,dx\,dt \\
&\hspace{0.4cm}+ 2\mathbb{E}\int_{t_1}^{t_2}\int_G v B_{22}\cdot\nabla v \,dx\,dt+\mathbb{E}\int_{t_1}^{t_2}\int_G (b_{21}q+b_{22}v)^2 \,dx\,dt.
\end{align*}
By Young's inequality, it is not difficult to see that
\begin{align}\label{ineq4.155fi}
\begin{aligned}
&\mathbb{E}\int_G \left[q^2(t_2) + v^2(t_2)\right] \,dx + \mathbb{E}\int_{t_1}^{t_2}\int_G |\nabla v|^2 \,dx\,dt\\
&\leq \mathbb{E}\int_G \left[q^2(t_1) + v^2(t_1)\right] \,dx + C\mathbb{E}\int_{t_1}^{t_2}\int_G \left(q^2 + v^2\right) \,dx\,dt.
\end{aligned}
\end{align}
For \(t\in[0,T]\), we define the energy functional
$$\mathcal{E}(t)=\mathbb{E}\int_G \left[q^2(t)+v^2(t)\right]\,dx.$$
Hence, inequality \eqref{ineq4.155fi} takes the form
\begin{align}\label{ineq4.155}
\begin{aligned}
&\mathcal{E}(t_2) + \mathbb{E}\int_{t_1}^{t_2}\int_G |\nabla v|^2 \,dx\,dt\leq \mathcal{E}(t_1) + C\int_{t_1}^{t_2} \mathcal{E}(t)\,dt.
\end{aligned}
\end{align}
\begin{itemize}
\item Firstly, by applying Gronwall's inequality to \eqref{ineq4.155}, we find that
$$\mathcal{E}(t_2) \leq C \,\mathcal{E}(t_1).$$
Then for any \( t \in (T/4, 3T/4) \), it follows that
\begin{align}\label{inee4.1666i}
\mathcal{E}\Big(t+\frac{T}{4}\Big) \leq C \,\mathcal{E}(t).
\end{align}
Integrating \eqref{inee4.1666i} with respect to \( t \in (T/4, 3T/4) \), we get
\begin{align}\label{estimmqp11}
\mathbb{E}\int_{T/2}^T\int_G \left(q^2 + v^2\right) \,dx\,dt 
\leq C\mathbb{E}\int_{T/4}^{3T/4}\int_G \left(q^2 + v^2\right) \,dx\,dt.
\end{align}
\item Secondly, from \eqref{ineq4.155}, we have that
\begin{align}\label{graditem}
\mathbb{E}\int_{T/2}^{T}\int_G |\nabla v|^2\,dx\,dt
\le \mathcal{E}\Big(\frac{T}{2}\Big)+
C \int_{T/2}^{T} \mathcal{E}(t)\,dt.
\end{align}
For the first term on the right-hand side of \eqref{graditem}, by the energy estimate, it is not difficult to see that for any $r\in(T/4,T/2)$,
\begin{align}\label{ineqqA}
\mathcal{E}\Big(\frac{T}{2}\Big)\leq C \int_{T/4}^{3T/4} \mathcal{E}(r) \,dr.
\end{align}
For the second term on the right-hand side of \eqref{graditem}, we have for any $s\in(T/2,T)$,
$$\mathcal{E}(s)\leq C\, \mathcal{E}\Big(s-\frac{T}{4}\Big).$$
It follows that
$$\int_{T/2}^{T}\mathcal{E}(s) ds\leq C \int_{T/2}^{T}\mathcal{E}\Big(s-\frac{T}{4}\Big) ds,$$
which gives that
\begin{align}\label{ineqqB}
\int_{T/2}^{T}\mathcal{E}(s) ds\leq C \int_{T/4}^{3T/4}\mathcal{E}(s) ds.
\end{align}
From \eqref{graditem}, \eqref{ineqqA} and \eqref{ineqqB}, we obtain that
\begin{align}\label{graditemsec}
\mathbb{E}\int_{T/2}^{T}\int_G |\nabla v|^2\,dx\,dt
\le  C \mathbb{E}\int_{T/4}^{3T/4}\int_G (q^2+v^2) \,dx\,dt.
\end{align}
\end{itemize}
Combining \eqref{estimmqp11} and \eqref{graditemsec}, we deduce that
\begin{align}\label{estimmqp11ne}
\mathbb{E}\int_{T/2}^T\int_G \left(q^2 + v^2+|\nabla v|^2\right) \,dx\,dt 
\leq C\mathbb{E}\int_{T/4}^{3T/4}\int_G \left(q^2 + v^2\right) \,dx\,dt.
\end{align}

On the other hand, for any \( t \in (0,T) \), computing \( d(h^2 + k^2) \), we have that
\begin{align*}
\begin{aligned}
\mathbb{E}\int_t^T\int_G d(h^2+k^2)\,dx
&= 2\mathbb{E}\int_t^T\int_G \sum_{i,j=1}^N \sigma_{ij}^1\frac{\partial h}{\partial x_i}\frac{\partial h}{\partial x_j} \,dx\,ds - 2\mathbb{E}\int_t^T\int_G a_{11} h^2 \,dx\,ds  \\
&\hspace{0.4cm} - 2\mathbb{E}\int_t^T\int_G a_{21} h k \,dx\,ds - 2\mathbb{E}\int_t^T\int_G b_{11} h H \,dx\,ds\\
&\hspace{0.4cm} - 2\mathbb{E}\int_t^T\int_G b_{21} hK \,dx\,ds- 2\mathbb{E}\int_t^T\int_G (h B_{11}+kB_{21})\cdot\nabla h \,dx\,ds \\
&\hspace{0.4cm} + \mathbb{E}\int_t^T\int_G H^2 \,dx\,ds   + 2\mathbb{E}\int_t^T\int_G \sum_{i,j=1}^N \sigma_{ij}^2\frac{\partial k}{\partial x_i}\frac{\partial k}{\partial x_j}  \,dx\,ds\\
&\hspace{0.4cm} - 2\mathbb{E}\int_t^T\int_G a_{12} h k \,dx\,ds   - 2\mathbb{E}\int_t^T\int_G a_{22} k^2 \,dx\,ds\\
&\hspace{0.4cm} 
- 2\mathbb{E}\int_t^T\int_G b_{12} k H \,dx\,ds- 2\mathbb{E}\int_t^T\int_G b_{22} k K \,dx\,ds\\
&\hspace{0.4cm}
- 2\mathbb{E}\int_t^T\int_G (h B_{12}+k B_{22})\cdot\nabla k \,dx\,ds+ \mathbb{E}\int_t^T\int_G K^2 \,dx\,ds\\
&\hspace{0.4cm}- 2\mathbb{E}\int_t^T\int_G \chi_{\mathcal{O}_1}  kv \,dx\,ds-2\mathbb{E}\int_t^T\int_G \chi_{\mathcal{O}_2} \nabla v\cdot\nabla k \,dx\,ds,
\end{aligned}
\end{align*}
which provides that
\begin{align}\label{ineqq1.18sec}
\begin{aligned}
-\mathbb{E}\int_t^T\int_G d(h^2+k^2) \,dx &\leq C\mathbb{E}\int_t^T\int_G (h^2+k^2+v^2+|\nabla v|^2) \,dx\,ds.
\end{aligned}
\end{align}
Since \( h(T,\cdot)=k(T,\cdot) = 0 \) in \( G \), a.s., using Gronwall's inequality for \eqref{ineqq1.18sec}, we derive that
$$\mathbb{E}\int_G [h^2(t)+k^2(t)] \,dx \leq C\mathbb{E}\int_t^T\int_G (v^2+|\nabla v|^2) \,dx\,ds.$$
Then for any \( t \in (T/2,T) \), it follows that
\begin{align}\label{estimmqp22}
\mathbb{E}\int_{T/2}^T\int_G (h^2+k^2)\,dx\,dt \leq C\mathbb{E}\int_{T/2}^{T}\int_G (v^2+|\nabla v|^2) \,dx\,dt.
\end{align}
Combining \eqref{estimmqp11ne} and \eqref{estimmqp22}, we arrive at 
\begin{align}\label{estimmqp}
\mathbb{E}\int_{T/2}^T \int_G (h^2 + k^2)\,dx\,dt \leq C \mathbb{E} \int_{T/4}^{3T/4} \int_G (q^2 + v^2) \,dx\,dt.
\end{align}
Next, it is straightforward to observe that
\begin{align*}
\mathbb{E}\int_{T/2}^T \int_G \textnormal{exp}(-Mt^{-1})(h^2 + k^2)\,dx\,dt \leq \mathbb{E}\int_{T/2}^T \int_G (h^2 + k^2)\,dx\,dt,
\end{align*}
which, combined with \eqref{estimmqp}, yields
\begin{align*}
\mathbb{E}\int_{T/2}^T \int_G \textnormal{exp}(-Mt^{-1})(h^2 + k^2)\,dx\,dt \leq C \mathbb{E}\int_{T/4}^{3T/4} \int_G \theta^2 \gamma^3 (q^2 + v^2) \,dx\,dt.
\end{align*}
Consequently,
\begin{align}\label{ineqq4.244i}
\mathbb{E}\int_{T/2}^T \int_G \textnormal{exp}(-Mt^{-1})(h^2 + k^2)\,dx\,dt \leq C \mathbb{E}\iint_Q \theta^2 \gamma^3 (q^2 + v^2) \,dx\,dt.
\end{align}
Applying the Carleman estimate \eqref{carlestcasca} to the right-hand side of \eqref{ineqq4.244i}, we deduce that
\begin{align}\label{secestim522}
\mathbb{E}\int_{T/2}^T \int_G \textnormal{exp}(-Mt^{-1})(h^2 + k^2)\,dx\,dt \leq C \mathbb{E}\iint_Q \left(h^2\chi_{G_0} + H^2 + K^2\right) \,dx\,dt.
\end{align}

On the other hand, it is easy to see that there exists a sufficiently large constant $M$ such that
$$\textnormal{exp}(-Mt^{-1})\leq C \theta^2\gamma^3,\quad \textnormal{in}\;(0,T/2)\times G,$$
which implies that
\begin{align}\label{estimmqpfir11}
\begin{aligned}
\mathbb{E}\int_0^{T/2}\int_G \textnormal{exp}(-Mt^{-1})(h^2+k^2)\,dx\,dt \leq C\mathbb{E}\iint_Q \theta^2\gamma^3 (h^2+k^2) \,dx\,dt.
\end{aligned}
\end{align}
Using again the Carleman estimate \eqref{carlestcasca} on the right-hand side of \eqref{estimmqpfir11}, we obtain
\begin{align}\label{firsesto518}
\begin{aligned}
\mathbb{E}\int_0^{T/2}\int_G \textnormal{exp}(-Mt^{-1})(h^2+k^2)\,dx\,dt \leq C \mathbb{E}\iint_Q \left(h^2\chi_{G_0} + H^2 + K^2\right) \,dx\,dt.
\end{aligned}
\end{align}
Finally, combining \eqref{secestim522} and \eqref{firsesto518}, we conclude the desired observability inequality \eqref{obseine4.9}. This completes the proof of Proposition \ref{propo5.1obseine}.
\end{proof}

Similarly to the proof of Proposition \ref{propo5.1obseine}. By using the Carleman estimate \eqref{carlestcasca33.4}, we prove the following observability inequality for the coupled system \eqref{adback4.77}-\eqref{adjoforr4.8} when $\beta\in(0,1)$.

\begin{prop}[Case $\beta\in(0,1)$]\label{propo5.1obseine2se}
We assume that \eqref{assA2}  holds. Then, there exist positive constants \( C \) and \( M \) depending only on \( G \), \( G_i \), \( \mathcal{O}_i \), \( T \), \( \sigma_{ij}^m \), \( \sigma_0 \), \( a_{ij} \), $b_{ij}$ and \( B_{ij} \) such that for any \( q_0, v_0 \in L^2_{\mathcal{F}_0}(\Omega;L^2(G)) \), the solution \( ((h,k;H,K);(q,v)) \) of the coupled system \eqref{adback4.77}-\eqref{adjoforr4.8} satisfies that
\begin{align*}
\begin{aligned}
\mathbb{E} \iint_Q \exp\left(-Mt^{-1}\right) \left( h^2 + k^2 \right) \, dx \, dt 
&\leq C \,\mathbb{E} \iint_Q \left[  h^2 \chi_{G_0}+k^2 \chi_{G_1} + H^2 + K^2 \right] \, dx \, dt.
\end{aligned}
\end{align*}
\end{prop}

\section{Proof of the Main Results}\label{sec5}
For completeness, this section is devoted to the proof of the main results of this paper, namely Theorems \ref{thmm1.3ins} and \ref{thmm1.3ins3}. We provide the detailed proof of Theorem \ref{thmm1.3ins}. The proof of Theorem \ref{thmm1.3ins3} follows by a similar argument, relying on the observability inequality established in Proposition \ref{propo5.1obseine2se}.

\begin{proof}[Proof of Theorem \ref{thmm1.3ins}]
Let \( \xi_1, \xi_2 \in L^2_\mathcal{F}(0,T; L^2(G)) \) satisfying \eqref{assonx1xi2ins}, and
consider the following linear subspace of the space 
\( \mathcal{U}_0 \):
\begin{align*}
\mathcal{Y} = \Big\{ &(h\chi_{G_0}, H, K) \mid \; ((h, k; H, K); (q, v)) \; \textnormal{is the solution of} \; \eqref{adback4.77}-\eqref{adjoforr4.8},\\
&
\quad \textnormal{with some} \; q_0, v_0 \in L^2_{\mathcal{F}_0}(\Omega;L^2(G)) \Big\}.
\end{align*}
We define the linear functional \(\mathcal{F}:\mathcal{Y}\to\mathbb{R}\) by
\[
\mathcal{F}(h\chi_{G_0}, H, K) = -\mathbb{E} \iint_Q (h\xi_1 + k\xi_2) \, dx \, dt.
\]
By the observability inequality \eqref{obseine4.9}, the functional \(\mathcal{F}\) is bounded on \(\mathcal{Y}\), and we have that
\begin{align}\label{esthanban}
||\mathcal{F}||_{\mathcal{L}(\mathcal{Y}; \mathbb{R})} \leq C \left( \left\| \textnormal{exp}\left(Mt^{-1}\right) \xi_1 \right\|_{L^2_{\mathcal{F}}(0,T;L^2(G))} + \left\| \textnormal{exp}\left(Mt^{-1}\right) \xi_2 \right\|_{L^2_{\mathcal{F}}(0,T;L^2(G))} \right).
\end{align}
Using the Hahn–Banach theorem, we can extend \( \mathcal{F} \) to a bounded linear functional on the whole space 
\(
\mathcal{U}_0,
\) and by the Riesz representation theorem, one can find controls $(u_1, u_3, u_4) \in \mathcal{U}_0$, such that
\begin{align}\label{fireq11}
-\mathbb{E} \iint_Q (h\xi_1 + k\xi_2) \, dx\, dt = \mathbb{E} \iint_Q (u_1 h \chi_{G_0}+ u_3 H + u_4 K) \, dx\, dt.
\end{align}
Furthermore, from \eqref{esthanban}, we also have that
\begin{align}\label{estfocontrll}
\begin{aligned}
\|(u_1,u_3,u_4)\|_{\mathcal{U}_0} \leq C \left( \left\| \text{exp}\left(Mt^{-1}\right) \xi_1 \right\|_{L^2_{\mathcal{F}}(0,T;L^2(G))} + \left\| \text{exp}\left(Mt^{-1}\right) \xi_2 \right\|_{L^2_{\mathcal{F}}(0,T;L^2(G))} \right).
\end{aligned}
\end{align}
Applying Itô’s formula to the systems \eqref{forr4.1}-\eqref{adback4.77}, we compute \( d(yh + zk) \), and by integrating the resulting expression over \( (0,T) \) and taking the expectation on both sides, we obtain
\begin{align}\label{eqq1.2}
\mathbb{E} \iint_Q \left( h\xi_1 + k\xi_2 + u_1 h\chi_{G_0}  + u_3 H + u_4 K - \chi_{\mathcal{O}_1} vz -\chi_{\mathcal{O}_2} \nabla v\cdot\nabla z \right) \, dx\, dt = 0.
\end{align}
Next, applying Itô’s formula to the systems \eqref{back45}-\eqref{adjoforr4.8}, we compute \( d(pq + rv) \) and derive that
\begin{align}\label{eqq1.232}
\mathbb{E}\int_G \left[ p(0) q_0 + r(0) v_0 \right] \, dx = \mathbb{E} \iint_Q (\chi_{\mathcal{O}_1} z v+\chi_{\mathcal{O}_2} \nabla v\cdot\nabla z ) \, dx\, dt.
\end{align}
From \eqref{eqq1.2} and \eqref{eqq1.232}, we find
\begin{align}\label{eqq6.44}
\mathbb{E}\int_G \left[ p(0) q_0 + r(0) v_0 \right] \, dx = \mathbb{E} \iint_Q \left( h\xi_1 + k\xi_2 + u_1 h \chi_{G_0}  + u_3 H + u_4 K \right) \, dx\, dt.
\end{align}
Combining \eqref{eqq6.44} and \eqref{fireq11}, we deduce that
\[
\mathbb{E}\int_G \left[ p(0) q_0 + r(0) v_0 \right] \, dx = 0.
\]
Since \( q_0 \) and \( v_0 \) are arbitrary functions of \(L^2_{\mathcal{F}_0}(\Omega;L^2(G)) \), we obtain that
\[
p(0,\cdot) = r(0,\cdot) = 0 \quad \text{in} \; G, \quad \textnormal{a.s.}
\]
Thus, by Proposition~\ref{proposs11} and the inequality \eqref{estfocontrll}, we conclude the proof of Theorem~\ref{thmm1.3ins}.
\end{proof}

\section*{Acknowledgements}
This article is based upon work from the project PRIN2022 D53D23005580006 ``Elliptic and parabolic problems, heat kernel estimates and spectral theory''. The third author is member of the ``Gruppo Nazionale per l’Analisi Matematica, la Probabilità e le loro Applicazioni (GNAMPA)'' of the Istituto Nazionale di Alta Matematica (INdAM).


\begin{thebibliography}{10}



\bibitem{surveyAmmarKBGT}
F. Ammar-Khodja, A. Benabdallah, M. González Burgos, and L. de Teresa.  
\newblock Recent results on the controllability of linear coupled parabolic problems: a survey.
\newblock{\em Mathematical Control and Related Fields,} \textbf{1} (2011), 267--306.

\bibitem{elgconvec23}
M. Baroun, S. Boulite, A. Elgrou, and L. Maniar.
\newblock Null controllability for stochastic parabolic equations coupled by first and zero order terms. \newblock {\em Applied Mathematics \& Optimization}, \textbf{91} (2025), 31.

\bibitem{elgrou1D23insensi24}
M. Baroun, S. Boulite, A. Elgrou, and O. Oukdach.
\newblock Insensitizing controls for stochastic parabolic equations with dynamic boundary conditions. 
\newblock{\em IMA Journal of Mathematical Control and Information}, \textbf{42} (2025).


\bibitem{bhann24}
K. Bhandari.
\newblock Insensitizing control problem for the Hirota–Satsuma system of KdV-KdV type. 
\newblock{\em Nonlinear Analysis}, \textbf{239} (2024), 113422.


\bibitem{ghadsanta24}
K. Bhandari and V. Hernández-Santamaría. 
\newblock Insensitizing control problems for the stabilized Kuramoto–Sivashinsky system. 
\newblock {\em ESAIM: Control, Optimisation and Calculus of Variations}, \textbf{30} (2024), 73.


\bibitem{bodafabre95}
O. Bodart and C. Fabre. 
\newblock Controls insensitizing the norm of the solution of a semilinear heat-equation. 
\newblock {\em Journal of Mathematical Analysis and Applications}, \textbf{195} (1995), 658--683.


\bibitem{BodarBurgosPerez2004}
O. Bodart, M. González-Burgos, and R. Pérez-García.
\newblock  Existence of insensitizing controls for a semilinear heat equation with a superlinear nonlinearity. 
\newblock {\em Communications in Partial Differential Equations}, \textbf{29} (2004), 1017--1050.





\bibitem{BodarBurgosPerez04NonAnalysis}
O. Bodart, M. González-Burgos, and R. Pérez-García.
\newblock Insensitizing controls for a heat equation with a nonlinear term involving the state and the gradient.
\newblock {\em Nonlinear Analysis: Theory, Methods \& Applications}, \textbf{57} (2004), 687--711.

	

\bibitem{BodaGnPer}
O. Bodart, M. González-Burgos, and R. Pérez-García.  
\newblock Insensitizing controls for a semilinear heat equation with a superlinear nonlinearity. 
\newblock {\em Comptes Rendus Mathematique}, \textbf{335} (2002), 677--682.






\bibitem{Preprielgr24jmaa}
S. Boulite, A. Elgrou, and L. Maniar.
\newblock Null controllability for cascade systems of coupled backward stochastic parabolic equations with one distributed control. 
\newblock{\em Journal of Mathematical Analysis and Applications}, \textbf{549} (2025), 129489.


\bibitem{bouetman25}
I. Boutaayamou, F. Et-tahri, and L. Maniar.
\newblock Insensitizing controls of a volume-surface reaction-diffusion equation with dynamic boundary conditions. 
\newblock{\em Nonlinear Differential Equations and Applications}, \textbf{32} (2025), 122.


\bibitem{boyhater19}
F. Boyer, V. Hernández-Santamaría, and L. de Teresa.
\newblock Insensitizing controls for a semilinear parabolic equation: a numerical 
approach.
\newblock{\em Mathematical Control and Related Fields}, \textbf{9} (2019), 117--158.


\bibitem{capfiltanka20}
R. D. A. Capistrano–Filho and T. Y. Tanaka. 
\newblock Controls insensitizing the norm of solution of a Schrödinger type system with mixed dispersion. 
\newblock{\em Journal of Differential Equations}, \textbf{416} (2025), 357--395.

\bibitem{Carl39}
T. Carleman. 
\newblock Sur un problème d’unicité pour les systèmes d’équations aux dérivées partielles à deux variables indépendantes.
\newblock{\em Arkiv för Matematik, Astronomi och Fysik}, \textbf{26} (1939), 1--9.



\bibitem{coron07}
J.-M. Coron.
\newblock Control and nonlinearity. 
\newblock{\em American Mathematical Society}, (2007).




\bibitem{calcarcerpa16}
B. M. Calsavara, N. Carreno, and E. Cerpa.
\newblock Insensitizing controls for a phase field system. 
\newblock{\em Nonlinear Analysis: Theory, Methods \& Applications}, \textbf{143} (2016), 120--137.

\bibitem{DapratoZabcz}
G. Da Prato and J. Zabczyk. 
\newblock Stochastic equations in infinite dimensions. 
\newblock{\em Cambridge university press}, 2014.

\bibitem{Djomkenn25}
L. Djomegne and C. Kenne. 
\newblock Insensitizing control of nonlinear coupled parabolic systems with a nonlocal spatial term. 
\newblock{\em Journal of Mathematical Analysis and Applications}, \textbf{558} (2026), 130359.




\bibitem{fernandez2006global}
E. Fern{\'a}ndez-Cara and S. Guerrero.
\newblock Global Carleman inequalities for parabolic systems and applications to controllability.
\newblock {\em SIAM Journal on Control and Optimization}, \textbf{45} (2006), 1395--1446.


\bibitem{elgomar26}
A. Elgrou and O. Oukdach. 
\newblock Multi-objective and hierarchical control for coupled stochastic parabolic systems.
\newblock {\em Numerical Algebra, Control and Optimization},  (2026).
 



 
\bibitem{fursikov1996controllability}
A. V. Fursikov and O. Yu. Imanuvilov.
\newblock Controllability of evolution equations. 
\newblock{\em Lecture Note Series 34, Research Institute of Mathematics, Seoul National University,} (1996).






\bibitem{gureSiam07}
S. Guerrero.
\newblock Null controllability of some systems of two parabolic equations with one control force. 
\newblock {\em SIAM Journal on Control and Optimization}, \textbf{46} (2007), 379--394.

\bibitem{kass20}
K. Kassab.
\newblock Negative and positive controllability results for coupled systems of second and fourth order parabolic equations. 
\newblock {\em Preprint}, (2020).



	
\bibitem{krylov1994}
N. V. Krylov.
\newblock A $W^n_2$ -theory of the Dirichlet problem for SPDEs in general smooth domains. 
\newblock{\em Probability Theory and Related Fields,} \textbf{98} (1994), 389--421.
 

\bibitem{kumamaj25}
M. Kumar and S. Majumdar. 
\newblock Insensitizing control problem for the Kawahara equation. 
\newblock{\em Nonlinear Differential Equations and Applications NoDEA}, \textbf{32} (2025), 1--38.







\bibitem{lions1989quel}
J.-L. Lions. 
\newblock Remarques préliminaires sur le contrôle des systemesa données incompletes.
\newblock {\em In Actas del Congreso de Ecuaciones Diferenciales y Aplicaciones (CEDYA), Universidad de Malaga}, (1989), 43--54.




	
	
\bibitem{liu2014global}
X. Liu.
\newblock Global Carleman estimate for stochastic parabolic equations, and its application.
\newblock{\em ESAIM: Control, Optimisation and Calculus of Variations}, \textbf{20} (2014), 823--839.


 

\bibitem{liu14couplfor}
X. Liu.
\newblock Controllability of some coupled stochastic parabolic systems with fractional order spatial differential operators by one control in the drift. 
\newblock{\em SIAM Journal on Control and Optimization,} \textbf{52} (2014), 836--860.


\bibitem{LiuuLiuX}
 L. Liu and X. Liu.
 \newblock Controllability and observability of some coupled stochastic parabolic systems. 
 \newblock{\em Mathematical Control and Related Fields}, \textbf{8} (2018), 829--854.

 
\bibitem{liu2019carleman}
X. Liu and Y. Yu.
\newblock Carleman estimates of some stochastic degenerate parabolic equations and application.
\newblock {\em SIAM Journal on Control and Optimization}, \textbf{57} (2019), 3527--3552.

\bibitem{luliu25}
Y. Lu and L. Liu. 
\newblock  Insensitizing controls for stochastic Kuramoto–Sivashinsky equation. 
\newblock {\em Mathematical Methods in the Applied Sciences}, \textbf{49} (2026), 6884–6897.

\bibitem{lu2011some}
Q. L{\"u}.
\newblock Some results on the controllability of forward stochastic heat equations with control on the drift.
\newblock{\em Journal of Functional Analysis}, \textbf{260} (2011), 832--851.

\bibitem{luZhang22mcrf}
Q. L{\"u} and X. Zhang.
\newblock A concise introduction to control theory for stochastic partial differential equations.
\newblock{\em Mathematical Control and Related Fields,} \textbf{12} (2022), 847--954.


\bibitem{lu2021mathematical}
Q. L{\"u} and X. Zhang.
\newblock Mathematical control theory for stochastic partial differential equations. 
\newblock{\em Springer
Nature Switzerland AG, Cham,} (2021).


\bibitem{msacar26}
R. Morales, M. C. Santos, and N. Carreño. (2026). 
\newblock Insensitizing the tangential gradient for reaction–diffusion equations with dynamic boundary conditions. 
\newblock{\em Journal of Optimization Theory and Applications}, \textbf{208} (2026), 55.

\bibitem{santa19}
M. C. Santos and T. Y. Tanaka. 
\newblock An insensitizing control problem for the Ginzburg–Landau equation. 
\newblock{\em Journal of Optimization Theory and Applications}, \textbf{183} (2019), 440--470.


\bibitem{omboelman25sec}
O. Oukdach, S. Boulite, A. Elgrou, and L. Maniar.
\newblock Stackelberg–Nash null controllability for stochastic parabolic equations. 
\newblock{\em Mathematical Methods in the Applied Sciences}, \textbf{48} (2025), 13164--13176.





\bibitem{tang2009null}
S. Tang and X. Zhang.
\newblock Null controllability for forward and backward stochastic parabolic equations.
\newblock {\em SIAM Journal on Control and Optimization}, \textbf{48} (2009), 2191--2216.


\bibitem{Tereza2000}
L. D. Teresa.  
\newblock Insensitizing controls for a semilinear heat equation: semilinear heat equation. 
\newblock{\em Communications in Partial Differential Equations}, \textbf{25} (2000), 39--72.

 
\bibitem{Tereza97Esaim}
L. D. Teresa. 
\newblock Controls insensitizing the norm of the solution of a semilinear heat equation in unbounded domains.
\newblock{\em ESAIM: Control, Optimisation and Calculus of Variations}, \textbf{2} (1997), 125--149.






\bibitem{Tereidenfication}
L. D. Teresa and E. Zuazua.
\newblock Identification of the class of initial data for the insensitizing control of the heat equation. 
\newblock{\em Communication on Pure and Applied Analysis}, \textbf{8} (2009), 457--471.




\bibitem{yansun2011}
Y. Yan and F. Sun.
\newblock Insensitizing controls for a forward stochastic heat equation. 
\newblock{\em  Journal of Mathematical Analysis and Applications}, \textbf{384} (2011), 138--150.


\bibitem{yuzhang23t}
Y. Yu and J.-F. Zhang.
\newblock Two multiobjective problems for stochastic degenerate parabolic equations.
\newblock{\em  SIAM Journal on Control and Optimization}, \textbf{61} (2023), 2708–2735.







 



\bibitem{observineqback}
D. Yang and J. Zhong. 
\newblock Observability inequality of backward stochastic heat equations for measurable sets and its applications. 
\newblock{\em SIAM Journal on Control and Optimization}, \textbf{54} (2016), 1157--1175.



\bibitem{zhanyingaodbc19}
M. Zhang, J. Yin, and H. Gao. 
\newblock Insensitizing controls for the parabolic equations with dynamic boundary conditions. 
\newblock{\em Journal of Mathematical Analysis and Applications}, \textbf{475} (2019), 861--873.

 




 \end{thebibliography}
\end{document}